\documentclass[a4paper, 12pt, oneside]{amsart}

\usepackage{amssymb}  % defines \nmid, for example
\usepackage{latexsym} % for \Box
\usepackage{color}
\usepackage{colonequals}
\usepackage{amsmath}
\usepackage{booktabs}
 \usepackage{epsfig}
\usepackage{tikz}

\usepackage{tikz}
\usepackage{tikz-cd}

\usepackage{setspace}
\definecolor{darkgreen}{rgb}{0,0.5,0}
\usepackage[
        colorlinks, citecolor=darkgreen,
        pdfauthor={Ingrid Bauer, Christian Gleissner},
        pdftitle={Rigid Manifolds},
        linktocpage        
]{hyperref}

\usepackage[alphabetic,backrefs,lite]{amsrefs} % for bibliography

\usepackage[all]{xy} % added "lite" to amsrefs and moved xy here; hope this works

\usepackage{enumerate} % for flexible labels in enumerated lists

\newtheorem{theorem}{Theorem}[section]
\newtheorem{proposition}[theorem]{Proposition}

\newtheorem{corollary}[theorem]{Corollary}
\newtheorem{lemma}[theorem]{Lemma}
\newtheorem{definition}[theorem]{Definition}
\newtheorem{remark}[theorem]{Remark}
\newtheorem{example}[theorem]{Example}

\newtheorem*{theorem-non}{Main Theorem}
\newtheorem*{question-non}{Question}

\theoremstyle{remark}
\newtheorem{rem}[theorem]{Remark}

\newenvironment{dedication}
  {%\clearpage           % we want a new page
   \thispagestyle{empty}% no header and footer
   \itshape             % the text is in italics
   \raggedleft          % flush to the right margin
  }
  {%\par % end the paragraph
  }

\newcommand\sB{{\mathcal B}}

\newcommand{\CC}{\ensuremath{\mathbb{C}}}
\newcommand{\RR}{\ensuremath{\mathbb{R}}}
\newcommand{\ZZ}{\ensuremath{\mathbb{Z}}}

\newcommand{\hol}{\ensuremath{\mathcal{O}}}

\newcommand{\PP}{\ensuremath{\mathbb{P}}}

\newcommand{\ra}{\ensuremath{\rightarrow}}

\DeclareMathOperator{\OO}{O}

\DeclareMathOperator{\Ext}{Ext}

\DeclareMathOperator{\Def}{Def}

\DeclareMathOperator{\Id}{Id}
\DeclareMathOperator{\Aut}{Aut}

\DeclareMathOperator{\Sing}{Sing}
\DeclareMathOperator{\diag}{diag}

\DeclareMathOperator{\GL}{GL}
\DeclareMathOperator{\SL}{SL}

\DeclareMathOperator{\Aff}{Aff}
\DeclareMathOperator{\Irr}{Irr}

\DeclareMathOperator{\He}{He}
\DeclareMathOperator{\ord}{ord}
\DeclareMathOperator{\lcm}{lcm}
 
  \DeclareMathOperator{\End}{End}
 
\numberwithin{equation}{section}

\newcounter{nootje}
\renewcommand\check[1]
  {\marginpar{\tiny\begin{minipage}{20mm}\begin{flushleft}\thenootje : #1\end{flushleft}\end{minipage}}\addtocounter{nootje}{1}}
\begin{document}
\title[Rigid Product Quotients  of Kodaira Dimension 0]{Towards a Classification of Rigid Product Quotient Varieties   of Kodaira Dimension 0 }

\author{Ingrid Bauer, Christian Gleissner}

\thanks{
The second author wants to thank A. Demleitner and D. Frapporti for useful discussions.}

\begin{abstract} 
In this paper the authors study quotients of the product of elliptic curves by a rigid diagonal action of a finite group $G$. It is shown that only for  $G = \He(3), \ZZ_3^2$, and only for dimension $\geq 4$ such an action can be free. A complete classification of the singular quotients in dimension 3 and the smooth quotients in dimension $4$ is given. For the other finite groups a strong structure theorem for rigid quotients is proven. \end{abstract}

\maketitle
\begin{dedication}
Dedicated to Fabrizio Catanese on the occasion of his 70th birthday with gratitude and admiration.
\end{dedication}

\tableofcontents

\section{Introduction}
A compact complex manifold is  called {\em rigid} if it has no nontrivial deformations. 
In \cite{rigidity} several notions  of rigidity have been discussed, the relations among them have been studied and many questions and conjectures have been proposed. 

We state here only the part of  \cite{rigidity}*{Definition 2.1}, which will be relevant for our purposes:
\begin{definition}\label{rigid} 
Let $X$ be a compact complex manifold of dimension $n$. 

1) A {\em deformation of $X$} is a  proper smooth holomorphic map of pairs $f \colon (\mathfrak{X},X)  \rightarrow (\mathcal{B}, b_0)$, 
where $(\sB,b_0)$ is a connected (possibly not reduced) germ of a complex space. 

2) $X$ is said to be  {\em  rigid}  if for each deformation of $X$,
$f \colon (\mathfrak{X},X)  \rightarrow (\sB, b_0)$
there is an open neighbourhood $U \subset \sB$ of $b_0$ such that $X_t := f^{-1}(t) \simeq X$ for all $t \in U$.

3)  $X$ is said to be  {\em infinitesimally rigid} if 
$H^1(X, \Theta_X) = 0$,
where $\Theta_X$ is the sheaf of holomorphic vector fields on $X$.
\end{definition}

\begin{rem}\label{kuranishi} 

1) If $X$ is infinitesimally rigid, then $X$ is also  rigid. This follows by Kodaira-Spencer-Kuranishi theory, since $H^1(X, \Theta_X)$ is the Zariski tangent space of the germ of analytic space which is the base $\Def(X)$ of the Kuranishi semiuniversal deformation of $X$.
So, if  $H^1(X, \Theta_X) =0$, $\Def(X)$ is a reduced point and all deformations are induced by the trivial deformation. 

The other implication does not hold in general  as it was shown in \cite{notinfinitesimally}, compare also \cite{kodairamorrow}.

2) Observe that, as it is shown in \cite[Theorem 2.3]{rigidity}, a compact complex manifold is rigid if and only if the base of the Kuranishi family $\Def(X)$  has dimension $0$.

3) The only rigid curve is $\PP^1$; for $n=2$ it was shown in \cite[Theorem 1.3]{rigidity} that a rigid compact complex surface has Kodaira dimension $- \infty$ or $2$.

\end{rem}

 That the restriction on the Kodaira dimension is a phenomenon in low dimensions and that in higher dimensions rigid manifolds are much more frequent has already been observed in \cite[Theorem 1.4]{rigidity} (cf. also \cite{beauville} for Kodaira dimension $0$,  \cite{BG} for Kodaira dimension $1$):
\begin{theorem} 
For all $n \geq 3$ and $-\infty,0 \leq k \leq n$ there is a infinitesimally rigid n-dimensional compact complex manifold $Z_{n,k}$ of Kodaira dimension $k$. 

\end{theorem}
 
One idea to construct the infinitesimally rigid examples is  to consider finite  quotients of smooth compact complex manifolds (often products of curves) with respect to a (infinitesimally) rigid holomorphic group action (see Definition \ref{infG}).
If one considers non free actions, under mild assumptions it is still true that the quotient is infinitesimal rigid (in dimension at least three), but since we are interested in infinitesimally rigid manifolds, we have to compare the infinitesimal deformations of the singular quotient with those of a suitable resolution of singularities.

The aim of this paper is to  give a classification of infinitesimally rigid quotients of a product of elliptic curves by a diagonal action of a finite group $G$ as far as  possible. We first prove the following:
\begin{theorem}\label{1}
Let $G$ be a finite group which admits  a rigid free diagonal action on a product of elliptic curves $E_1 \times \ldots \times E_n$. Then:
\begin{enumerate}
\item $n\geq 4$,
\item $E_i$ is the Fermat elliptic curve for each $1 \leq i \leq n$,
\item $G = \ZZ_3^2$ or the {\em Heisenberg group} $\He(3)$ of order $27$.
\end{enumerate}

\end{theorem}

These two groups are part of  four so-called  {\em exceptional} groups, which are groups admitting a rigid action on an elliptic curve such that the translation part is not uniquely determined (cf. Proposition \ref{uniquetrans}).

For the non exceptional case instead we have the following result:

\begin{theorem}\label{2}
Assume that $G$ is not exceptional and admits a rigid diagonal action on a product of elliptic curves $E_1 \times \ldots \times E_n$. Then the quotient  
 is isomorphic to $X_{n,d}:=E^n/\mathbb Z_d$, where $\mathbb Z_d$ acts on $E^n$ by multiplication with $\zeta_d \cdot \Id$. Here $d =3,6$ and $E$ is the Fermat elliptic curve, or $d=4$ and $E$ is the harmonic elliptic curve.
\end{theorem}

\begin{rem}
1) Observe that for the non exceptional groups we have in each  dimension exactly one case for $d$. Unfortunately these quotients are singular, the singularities are of type $\frac{1}{k}(1, \ldots, 1)$, where $k~ \big\vert ~ d$.  For $n \geq d$ these singularities are canonical and it can be shown that there is a resolution of singularities which is infinitesimally rigid (cf. \cite{BG}).

2) For $n=d=3,4,6$ the variety $X_{n,d}$ is a singular Calabi-Yau variety (cf. \cite{beauville}).
\end{rem}

We finally completely classify rigid diagonal actions of the two exceptional groups $\mathbb Z_3^2$ and $\He(3)$ in dimensions $3$ and the free ones in dimension $4$. In fact, we prove the following:

\begin{theorem}\label{3} 
1) There is exactly  one 
isomorphism class of quotient manifolds 
$E^4/\mathbb Z_3^2$ resp. $E^4/\He(3)$ obtained by a rigid free and diagonal action. They  have non isomorphic fundamental groups. 

2) For each exceptional group $\mathbb Z_3^2$ and $\He(3)$ there are exactly 
four isomorphism classes of (singular) quotients 
$X_i:=E^3/\mathbb Z_3^2$ and $Y_i:=E^3/\He(3)$ 
 obtained by a rigid diagonal $G$-action:
 \begin{itemize}
\item[i)] 
 $X_4$ and $Y_4$ are isomorphic to Beauville's Calabi-Yau threefold $X_{3,3}$. 
\item[ii)]   $X_3$ and $Y_3$ are also Calabi-Yau, uniformized by $X_{3,3}$ and  
admit  crepant resolutions, which are rigid. 
\item[iii)]  
$X_2$ and $X_3$, resp. $Y_2$ and $Y_3$,  are diffeomorphic but not biholomorphic. 
\item[iv)]  The eight threefolds  $X_i$, $Y_j$  form five distinct  topological classes.
\end{itemize}

\end{theorem}

Our paper is organized as follows:

In a short first section we recall some facts on rigid group actions on complex manifolds. In the second chapter we treat rigid actions on elliptic curves recalling that a finite group $G$ admitting a rigid action on an elliptic curve $E$ has to be of the form $G= A \rtimes_{\varphi_d} \mathbb Z_d$, where $d \in \{3,4,6\}$, and $E$ is the Fermat or the harmonic elliptic curve. Moreover we show that if $G$ admits a rigid action on a product of elliptic curves, these curves have all to be the same.
In the following section we prove Theorems \ref{1} and \ref{2}.

The last section is dedicated to the classification of the actions of the exceptional groups in dimensions $3$ and $4$. The results on the isomorphism classes in Theorem  \ref{3} is done using a MAGMA algorithm \cite{MAGMA}, which is available on the website of the second author:
\begin{center}
\url{http://www.staff.uni-bayreuth.de/~bt300503/publi.html}.
\end{center}
 To obtain the  topological results we exploit  the fact that the fundamental groups of the smooth quotients resp. the orbifold fundamental groups  of the singular quotients in dimension three are  crystallographic groups.

\section{Rigid group actions on complex manifolds}

\noindent 

In this short introductory section we recall the definition of rigid group actions on complex manifolds and give a criterion for the rigidity of the diagonal action on a product of such. 

Since in all of this article we only discuss {\em infinitesimal} rigidity, we often drop the adverb {\em infinitesimally} and only talk about rigidity.

\begin{definition}\label{infG}
Let $X$ be a compact complex manifold and $G$ be a finite group acting holomorphically  on $X$. We say that the $G$-action is {\em infinitesimally rigid} if and only if 
$H^1(X,\Theta_X)^G=0$, where $\Theta_X$ is the tangent sheaf of $X$. 
\end{definition}

\begin{rem} \

1) If the action of $G$ is not faithful, then, denoting the kernel by $K$, we obviously have that $H^1(X,\Theta_X)^G=H^1(X,\Theta_X)^{G/K}$, hence after replacing $G$ by $G/K$ we may assume that $G$ acts faithfully.

2) If $G$ acts freely in codimension one, then for all $i$  there is an isomorphism
$$
H^i(X/G,\Theta_{X/G}) \simeq H^i(X,\Theta_X)^G. $$
In particular, since $H^1(X/G,\Theta_{X/G})$ classifies the infinitesimal equisingular deformations of $X/G$, we see that, if the action is rigid, the quotient $X/G$ has no equisingular deformations. 

3) Consider the low term exact sequence of the local-to-global $\Ext$-spectral sequence on $Z := X/G$:
$$
0 \to H^1(Z, \Theta_Z) \to \Ext^1(\Omega^1_Z, \hol_Z) \to H^0(Z,\mathcal{E}xt^1(\Omega^1_Z, \hol_Z)) \to \ldots. 
$$
By Schlessinger's result \cite{schlessinger}  isolated quotient singularities in dimensions at least $3$ are infinitesimally rigid, i.e., $\mathcal{E}xt^1(\Omega^1_Z, \hol_Z) =0$. Therefore the above exact sequence shows that, if $G$ has only isolated fixed points on $X$, then
$$
H^1(Z, \Theta_Z) \simeq \Ext(\Omega^1_Z,\hol_Z). 
$$
This means that for $Z$ all infinitesimal deformations are equisingular. In particular, for showing that $Z$ is an infinitesimally rigid variety, it suffices to show that $H^1(X, \Theta_X)^G = 0$.

4) If $Z$ is singular, one has to make sure that there is a resolution of singularities $\widehat{Z} \ra Z$ such that $\widehat{Z}$ is infinitesimally rigid, since primarily we are interested in rigid {\em manifolds} (cf. Proposition \ref{resolution}).

\end{rem}

Let $G$ be a finite group acting holomorphically on the compact complex manifolds $X_1, \ldots, X_n$, then the diagonal subgroup 
$\Delta_G \simeq G$ acts in a natural  way  on the product $X_1 \times \ldots \times X_n$, in fact setting:
$$
g(x_1, \ldots, x_n):=(g x_1, \ldots, g x_n).
$$
It is natural to call this action   {\em the diagonal} $G$-action on the product $X_1 \times \ldots \times X_n$. 
K\"unneth's formula allows us to give a reformulation of the rigidity of the diagonal action in terms of the individual actions on the factors: 

\begin{proposition}\label{rigiddiag}
Let $G$ be a finite group acting holomorphically on  the compact complex manifolds $X_1, \ldots, X_n$. Then,  the diagonal action on
 $X:=X_1 \times \ldots \times X_n$ is (infinitesimally) rigid, if and only if: 
\begin{enumerate}
\item
the $G$-action on each $X_i$ is rigid and 
\item 
$\big[H^0(X_i,\Theta_{X_i}) \otimes H^1(X_j,\mathcal O_{X_j})\big]^G=0$ for all $i \neq j$.
\end{enumerate}
\end{proposition}

\begin{proof}
Let  $p_i \colon X \ra X_i$ be the projection onto the $i$-th factor. Then 
\[
H^1(X,\Theta_{X})=\bigoplus_{i=1}^{n}H^1(X, p_i^{\ast}\Theta_{X_i}). 
\]

For each $1 \leq i \leq n$ we have 
\[
H^1(X, p_i^{\ast}\Theta_{X_i})= H^1(X_i,\Theta_{X_i}) \oplus \bigg[H^0(X_i,\Theta_{X_i}) \otimes \bigg(\bigoplus_{\stackrel{j=1}{j \neq i}}^{n}
H^1(X_j, \mathcal O_{X_j}) \bigg)\bigg], 
 \]
according to K\"unneth's formula. The claim follows, by taking the $G$-invariant part. 
\end{proof}

\begin{rem} 
%\begin{enumerate}
%\item
In the special case where the products $h^0(X_i,\Theta_{X_i}) \cdot h^1(X_j,\mathcal O_{X_j})$ vanish, a diagonal $G$-action is rigid if and only if the $G$-action
 on each factor $X_i$ is rigid. This happens for example if the complex manifolds $X_i$ are regular or of general type. 
% \item
% It is well known that a holomorphic $G$-action on a compact Riemann surface $X$ is rigid if and only if $X/G \simeq \mathbb P^1$ and the quotient map
% $$\pi \colon X \to X/G$$ is branched in three points $p_1$, $p_2$ and $p_3$. 
% In this case we shall say that $\pi$ is a \emph{triangle cover}. 

% \item
% Let $\pi \colon X \to X/G \simeq \mathbb P^1$ be a triangle cover, 
%$G_{x_i}$ be the stabilizer of a point $x_i \in \pi^{-1}(p_i)$ and $m_i=|G_{x_i}|$, then 
% famous Riemann Hurwitz formula holds:
%\[
%2g(X)-2=|G|\cdot \bigg(1 - \frac{1}{m_1} - \frac{1}{m_2} - \frac{1}{m_3}\bigg). 
%\]
%The triple $[m_1,m_2,m_3]$ is called the branching signature of $\pi$. 
%If $g(X)=1$ i.e. if $X=E$ is an elliptic curve, then  only  three branching signatures are possible:
%$$
%[3,3,3], \ [2,4,4], \   {\rm or}   \ [2,3,6].

%\end{enumerate}
\end{rem}

\section{Rigid actions on  elliptic curves}

In this paragraph  we study rigid diagonal $G$-actions on a product $E_1 \times \ldots \times E_n$ of elliptic curves under the additional assumption that $G$ acts faithfully on each factor.

Recall that any holomorphic map between elliptic curves, or more generally between complex tori, is induced by an affine linear map. Since the tangent bundle of an elliptic curve is  trivial, Proposition \ref{rigiddiag} has a particularly  simple   reformulation:

\begin{proposition}\label{quadraticdiff}
Let $G$ be a finite group acting holomorphically on  the elliptic curves  $E_1, \ldots, E_n$, then the following are equivalent:
\begin{enumerate}
\item
the diagonal $G$-action on
 $E_1 \times \ldots \times E_n$ is rigid, 
\item
none of the  quadratic differentials  $dz_i \otimes dz_j $  is  $G$-invariant. 
\end{enumerate}
\end{proposition}

\begin{proof}
By duality, the rigidity conditions in Proposition \ref{rigiddiag} are  equivalent to 
\[
H^0(E_i,(\Omega_{E_i}^1)^{\otimes 2})^G =0 \qquad \makebox{and} \qquad  
\big[H^1\big(E_i, (\Omega_{E_i}^1)^{\otimes 2}\big) \otimes H^0(E_j, \Omega_{E_j}^1)\big]^G  = 0. 
\]
We use  that $H^0(E_i,(\Omega_{E_i}^1)^{\otimes 2}) = \langle dz_i^{\otimes 2} \rangle $ and 
$$H^1\big(E_i, (\Omega_{E_i}^1)^{\otimes 2}\big)  \simeq H_{\overline{\partial}}^{1,1}(E_i,\Omega_{E_i}^1) = \langle (dz_i \wedge d\overline{z}_i) \otimes dz_i \rangle,$$ 
to  rewrite the rigidity conditions  as: 
\[
\langle dz_i^{\otimes 2} \rangle^G =0 \qquad \makebox{and} \qquad  
 \langle (dz_i \wedge d\overline{z}_i) \otimes dz_i \otimes dz_j \rangle^G  = 0. 
\]
The claim follows since $G$ acts trivially on each $2$-form $dz_i \wedge d\overline{z}_i$.
\end{proof}

\begin{rem}\label{canrep}  
1) A holomorphic $G$-action of a finite group on an elliptic curve $E$   induces a natural one-dimensional representation 
\[
\varphi_E \colon G \to \GL\big(H^0(E, \Omega_{E}^1) \big), \qquad g \mapsto [dz \mapsto (g^{-1})^{\ast}dz ].
\]
It is  called  the \emph{canonical representation}. As it is one dimensional we identify it with its character 
$\chi_E$. If the linear part of  $g\in G$ is $a$, then  the value of $\chi_E(g)$ is equal to $\overline{a}$.

2) In terms of the characters $\chi_{E_i}$ the statement on the quadratic differentials $dz_i\otimes dz_j$ in Proposition \ref{quadraticdiff} translates to 
$\chi_{E_i} \cdot \chi_{E_j} \neq \chi_{triv}$, for all $1 \leq i,j\leq n$. 

3) It is well known that a $G$-action on an elliptic curve $E$ is rigid  if and only if 
$E/G\simeq \mathbb P^1$ and the Galois cover $\pi \colon E \to E/G\simeq \mathbb P^1$ is branched in three points $p_1$, $p_2$ and $p_3$. 
In this case the branching signatures $m_i$, which are defined to be the orders of the (cyclic) stabilizer groups $G_{x_i}$ for some  $x_i \in \pi^{-1}(p_i)$,  
are up to permutation equal to:
\[
[m_1,m_2,m_3]=[3,3,3], \quad [2,4,4], \quad \makebox{or} \quad [2,3,6], 
\]
see \cite[Chapter III, Lemma 3.8 b)]{miranda}. 

\end{rem}

Rigid group actions on compact Riemann surfaces  can be described purely in terms of group theory by \emph{Riemann's existence theorem}. 
In our situation we only need the following (much weaker) version of this theorem for elliptic curves.

\begin{proposition}\label{RET}
A finite group $G$ admits a rigid action  on an elliptic curve $E$ if and only if there are elements $g_1,g_2,g_3 \in G$ of order $m_i=\ord(g_i)$, which generate $G$, fulfill the relation 
$g_1g_2g_3=1_G$ and $[m_1,m_2,m_3]=[3,3,3]$, $[2,4,4]$ or $ [2,3,6]$.
\end{proposition}

We refer to  \cite[Chapter III]{miranda} for details and mention only the following:  

\begin{rem}\label{RemRET} 
1) A triple of elements $V:=[g_1,g_2,g_3]$ like  in the Proposition  above is called a generating triple of $G$ with (branching) signature $[m_1,m_2,m_3]$.
\item 
Assume that  $E$  admits a rigid $G$-action, let 
$\pi \colon E \to E/G\simeq \mathbb P^1$ be the quotient map and let $\mathcal B:=\lbrace p_1, p_2, p_3 \rbrace$ be the set of branch points of $\pi$. 
The fundamental group of $\mathbb P^1\setminus \mathcal B$ is generated by three simple loops $\gamma_i$ around $p_i$ fulfilling a single relation 
$\gamma_1 \ast \gamma_2 \ast \gamma_3=1$. 
The  elements $g_i$ are obtained as the images of $\gamma_i$ under the monodromy homomorphism 
\[
\eta \colon \pi_1(\mathbb P^1\setminus \mathcal B, p) \to G. 
\]

2) The cyclic subgroups $\langle g_i \rangle$ and their conjugates provide the non-trivial stabilizer groups  of the $G$-action on $E$. Let $x_i \in E$ be a point with stabilizer $G_{x_i}=\langle g_i \rangle$, then  $g_i$  acts around $x_i$  as a rotation by 
\[
\exp\bigg(\frac{2\pi \sqrt{-1}}{m_i} \bigg). 
\]
This rotation constant is nothing else than  the  linear part of any affine transformation inducing $g_i$. 
In particular, the character $\chi_E$ (cf.  Remark  \ref{canrep}) can be read of directly  from the generating triple. 

3) The union of the non-trivial stabilizer groups will be denoted by $\Sigma_V$. Observe that  $\Sigma_V$ consists of the identity and all group elements which are not  translations.

\end{rem}

The further discussion heavily relies on the structure of the automorphism group of an elliptic curve $E$. 
Recall that $\Aut(E)$ is a semidirect product 
\[
\Aut(E)=E \rtimes \Aut_0(E), 
\]
where $\Aut_0(E) \simeq \mathbb Z_2$, $\mathbb Z_4$ or $\mathbb Z_6$.

\begin{proposition}\cite[Chapter III Proposition 1.12]{miranda} 
An elliptic curve with $\Aut_0(E)\simeq \mathbb Z_4$ is isomorphic to 
 $\mathbb C/ \mathbb Z[i]$. An elliptic curve with $\Aut_0(E)\simeq \mathbb Z_6$ is isomorphic to 
$\mathbb C/ \mathbb Z[\zeta_3]$. 
Here,  $ \mathbb Z[i]$ are the {\em Gaussian} and $\mathbb Z[\zeta_3]$ the {\em Eisenstein integers}. 
\end{proposition} 
  
Assume now that $E$ admits a rigid faithful  action of a finite group  $G$. 
Then by   Remark \ref{canrep} (3), the group  $G$ has an  element of order $3,4$ or $6$ with a fixed point which up to translation is the origin of $E$.  
Whence $\Aut_0(E)\simeq \mathbb Z_4$ or $\mathbb Z_6$ and  $G$ must be also a semidirect product 
\[
G\simeq A\rtimes \mathbb Z_d,
 \]
with $d \in \{3,4,6\}$. The normal subgroup $A$ can be considered as a subgroup of  the group of $n$-torsion points 
$E[n] \simeq \mathbb Z_n^2$ for a suitable integer $n$. 

More precisely, we have the following:

\begin{proposition}\label{wallgrps}
A finite group $G$ admits a faithful rigid holomorphic action on an elliptic curve $E$, if and only if it is isomorphic to a semidirect product 
\[
A \rtimes_{\varphi_d} \mathbb Z_d, 
\]
where $d=3,4$ or $6$ and $A \leq \mathbb Z_n^2$ is a subgroup for some $n$,  invariant under the action 
\[
\varphi_d \colon  \mathbb Z_d \to \Aut (\mathbb Z_n^2), 
\]
defined  by:
\begin{itemize}
\item
$\varphi_3(1)(a,b)=(-b,a-b)$, 
\item
$\varphi_4(1)(a,b)=(-b,a)$ or
\item
$\varphi_6(1)(a,b)=(-b,a+b)$.
\end{itemize}
The possible  branching signatures $[m_1,m_2,m_3]$ of the cover $\pi \colon E \to E/G$, the  abelianisations of $G$ and the isomorphism types 
of $E$ are summarised in the table below:

{\begin{center}
\begin{tabular}{ c  c  c  c  }
  & $ $d=3$ ~ $ & $ ~ $d=4$ ~ $  & $ ~ $d=6$~ $   \\
 \hline \hline 
 $[m_1,m_2,m_3]$ & $ \quad   [3,3,3]  \quad   $ & $ \quad  [2,4,4]  \quad  $  & $ \quad  [2,3,6] \quad  $   \\

$G^{ab}$  & $ \quad \mathbb Z_3  $ or $ \mathbb Z_3^2   \quad $ & $ \quad \mathbb Z_4  $ or  $ \mathbb Z_2 \times \mathbb Z_4  \quad $  & $  \quad  \mathbb Z_6  \quad $  \\

  $E$ &  $\mathbb C/\mathbb Z[\zeta_3]$ & $\mathbb C/ \mathbb Z[i]$  &  $\mathbb C/ \mathbb Z[\zeta_3]$  \\
\hline
%\hline
  \end{tabular}
  \end{center}
  }

\end{proposition}

Before we proof the above proposition we observe the following:
\begin{rem}\label{tworems}

1) By Proposition  \ref{RET} a group $G=A \rtimes_{\varphi_d} \mathbb Z_d$ with a rigid action on an elliptic curve $E$ is a quotient of the  triangle group 
\[
\mathbb T(m_1,m_2,m_3)=  \langle a,b,c  ~ | ~ a^{m_1} = b^{m_2} =c^{m_3}= abc=1 \rangle,
\]
where $[m_1,m_2,m_3]$, $d$, and $E$ are according to the table above. Whence $G^{ab}$  is  a quotient of 
$\mathbb T(m_1,m_2,m_3)^{ab}$. Since  the canonical projection  $p \colon G \to \mathbb Z_d$ induces a surjective homomorphism 
$p^{ab} \colon G^{ab} \to \mathbb Z_d$, the isomorphism types of $G^{ab}$ in the table follow  from the well known formula 
$\mathbb T(m_1,m_2,m_3)^{ab} \simeq  \mathbb Z/k_1 \times \mathbb Z/k_2$,
where  
\begin{itemize}
\item  $k_1:=\gcd(m_1,m_2,m_3)$,   
\item  $k_2:=\lcm\big(\gcd(m_1,m_2), \gcd(m_2,m_3), \gcd(m_1,m_3)\big)$.
 \end{itemize}
2) It follows immediately from the above table that if a finite group $G$ admits a rigid action on an elliptic curve $E$, the isomorphism type of $E$ as well as  the branching signature 
of the cover $E\to E/G$ is uniquely determined already by the abelianisation of $G$. 

\end{rem}

An immediate geometric consequence of  Remark \ref{tworems} (2) is: 
\begin{corollary}\label{alliso}
Let $G$ be a finite group with a rigid  diagonal  action on a product of elliptic curves $E_1 \times \ldots \times E_n$, then the curves are all isomorphic to 
$\mathbb C/\mathbb Z[i]$ or they are all isomorphic to $\mathbb C/\mathbb Z[\zeta_3]$. 

Moreover, the branching signature $[m_1,m_2,m_3]$ is the same for each cover $\pi_i \colon E_i \to E_i/G$. 
\end{corollary}

\begin{proof}[Proof of Proposition \ref{wallgrps}.]
The formulas for $\varphi_d$ are immediately derived  from the semidirect product structure of $\Aut(E)\simeq E \rtimes \Aut_0(E)$.  
The claim about the abelianisations follows from Remark \ref{tworems}  (2).

We are left to show  that any group $A \rtimes_{\varphi_d} \mathbb Z_d$ has a rigid action on 
$E=\mathbb C/\mathbb Z[i]$ or 
$\mathbb C/\mathbb Z[\zeta_3]$. For this purpose, we consider the natural action 
\[
A \rtimes_{\varphi_d} \mathbb Z_d \to \Aut(E), \qquad (a,b,c) \mapsto \bigg [z \mapsto \zeta_d^cz + \frac{a + \zeta_d b}{n} \bigg], 
\]
which is clearly rigid, because 
$A \rtimes_{\varphi_d} \mathbb Z_d$ acts non trivially on the generator $dz^{\otimes 2}$ 
of $H^0(E,\omega_E^{\otimes 2})$.
\end{proof}

We end the section by the following useful result.
\begin{lemma}\label{orderel}
The order of an element of $A \rtimes_{\varphi_d} \mathbb Z_d$, which is not contained in $A$, is equal to the order of its image under 
the canonical projection $ p \colon A \rtimes_{\varphi_d} \mathbb Z_d \to \mathbb Z_d$.  
 \end{lemma}
\begin{proof}
Let $f$ be such an element. Then the order of $p(f)$  divides the order of $f$ and  it suffices to show the following: 
let $\omega_k$ be a $k$-th primitive root of unity and $f(z)=\omega_kz+b$, then $f^k=id$. This claim follows immediately from the well known formula 
$\omega_k^{k-1}+ \ldots+\omega_k+1=0$ and the computation  
\[
f^k(z)=z+ (\omega_k^{k-1}+ \ldots+\omega_k+1)b=z.
\] 
\end{proof}

\section{Classifying rigid diagonal actions on products of elliptic curves}

In this section we shall show that there are only four candidates of finite groups, which we will call {\em exceptional groups} that may allow 
a free rigid action on the product of (at least 4) elliptic curves,  and on the other hand we shall see that all other groups admit in each dimension exactly one rigid (singular) quotient.

%We first  consider the  question of the existence of a  {\em free} rigid diagonal  $G$-action on a product of elliptic curves $E_1 \times \ldots \times E_n$.  
%The quotients of products of such actions are examples of  infinitesimally rigid projective manifolds of Kodaira dimension zero. The existence of such examples for dimension $n \geq 4$ 
%has first been shown in \cite[Theorem 3.4]{rigidity}. \check{Das eventuell raus,  passiert ja im n\"achsten Abschnitt}

The following is a trivial but useful observation:
\begin{rem}
Since we assume the action of $G$ on $E_1 \times \ldots \times E_n$  to be diagonal,  it is free if and only if  for each $g \in G$ there exists an index  $1 \leq j \leq n$ such that $g$ acts on $E_j$ as a translation. 
\end{rem}
 This motivates the following definition:
\begin{definition}
Let $\psi \colon G \to \Aut(E)$ be a faithful action of a finite group on an elliptic curve.  We define the {\em translation group of} $\psi$ to be 
\[
T_{\psi} :=\lbrace g \in G ~ \big\vert ~ \psi(g) ~\makebox{is a translation} \rbrace. 
\]
\end{definition}

\begin{remark}\label{transgroup}
Let $G$ be a finite group.  

1) The translation group $T$ of an action $\psi \colon G \to \Aut(E)$ is a normal abelian subgroup of $G$. If the action is moreover rigid, then  $G \simeq T \rtimes  \mathbb Z_d$, where 
$d=3,4$ or $6$. 

2) Let $G$  be a finite group, admitting  a diagonal action on a product of elliptic curves $E_1 \times \ldots \times E_n$. Then the action is free if and only if 
\[
G = T_1 \cup \ldots \cup T_n, 
\]
where $T_i$ is the translation group of the action on the i-th factor.  
\end{remark}

A necessary condition for a group $A \rtimes_{\varphi_d} \mathbb Z_d$ to allow a rigid and free action on a product of elliptic curves is the existence of more than one 
 normal abelian subgroup with quotient  $\mathbb Z_d$. It turns out that there are only four of them  and in fact in all other cases the translation group of any rigid action is $A$.

\begin{proposition}\label{uniquetrans}
Let $G=A \rtimes_{\varphi_d} \mathbb Z_d$  be a finite group and $\psi \colon G \ra \Aut(E)$ be a rigid action on an elliptic curve $E$.  Then  $T_{\psi} = A$  except if $G$ is one of the following: 
\[
\mathbb Z_3^2, \quad  \mathbb Z_3^2 \rtimes_{\varphi_3}  \mathbb Z_3, \quad   \mathbb Z_2 \times \mathbb Z_4 \quad \makebox{or} \quad   \mathbb Z_2^2 \rtimes_{\varphi_4}  \mathbb Z_4.
\]
\end{proposition}
These four groups we shall call {\em exceptional}.
%\begin{rem}
%i) Both  exceptional groups have  {\em four} normal abelian subgroups with factor group $\mathbb Z_3$.
%
%ii) None of the exceptional groups admit a rigid action on a curve $C$ of genus $g(C) \geq 2$. \check{correct the error in "Beauville surfaces"!}
%\end{rem}
Before we give a proof of Proposition \ref{uniquetrans}, we recall some structural properties  of  the exceptional groups:

\begin{rem}\label{structureHeis} \

(1) We denote the group  $\mathbb Z_3^2 \rtimes_{\varphi_3}  \mathbb Z_3$ by 
$\He(3)$, because it is isomorphic to the Heisenberg group  of 
upper triangular matrices
\[
\begin{pmatrix}
1 & a & c   \\
0 & 1 & b \\
0 & 0 & 1 \\
\end{pmatrix} \in \SL(3,\mathbb F_3). 
\]
The normal abelian subgroups $A_i \trianglelefteq \He(3)$ of index three  are the 
preimages of the index three subgroups of  $\mathbb Z_3^2$ under the surjective homomorphism 
\[
\alpha \colon \mathbb Z_3^2 \rtimes_{\varphi_3}  \mathbb Z_3 \to \mathbb Z_3^2, \qquad (a,b,c) \mapsto (a+b,c). 
\] 
They are all isomorphic to $\mathbb Z_3^2$ and the intersection of two of them equals the kernel of $\alpha$. 
Moreover, the kernel,  the center $C_3:=Z\big(\He(3)\big)\simeq \mathbb Z_3$ and the  commutator subgroup $[\He(3),\He(3)]$ are all the same. 

The groups $A_i$, together with the center  $C_3$ are all non-trivial  normal subgroups of $\He(3)$: 

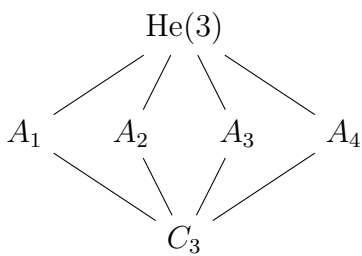
\begin{center}
\begin{tikzpicture}[scale=.7]
  \node (one) at (0,2) {$\He(3)$};
  \node (a) at (-3,0) {$A_1$};
  \node (b) at (-1,0) {$A_2$};
  \node (c) at (1,0) {$A_3$};
  \node (d) at (3,0) {$A_4$};
  \node (zero) at (0,-2) {$C_3$};
  \draw (zero) -- (a) -- (one) -- (b) -- (zero) -- (c) -- (one) -- (d) -- (zero);
\end{tikzpicture}
\end{center}

Since the center $C_3$ is contained in all $A_i$'s, Proposition \ref{transgroup} (1) implies that it  acts by translations with respect to any rigid action 
$\psi \colon \He(3) \to \Aut(E)$.

(2) There are two normal subgroups of  $\mathbb Z_2^2 \rtimes_{\varphi_4}  \mathbb Z_4$ with quotient $\mathbb Z_4$. They are isomorphic to $\mathbb Z_2^2$ and 
obtained as the preimages of the corresponding subgroups $\langle (1,0) \rangle $ and 
$\langle (1,2) \rangle $ of $\mathbb Z_2 \times \mathbb Z_4$ under the surjective homomorphism 
\[
\beta \colon \mathbb Z_2^2 \rtimes_{\varphi_4}  \mathbb Z_4  \to \mathbb Z_2 \times \mathbb Z_4, \qquad (a,b,c) \mapsto (a+b,c).  
\] 
It follows immediately from Remark \ref{transgroup}, (2) that none of these groups can admit a free diagonal action on the product of elliptic curves.
\end{rem}

\begin{proof}[Proof of Proposition \ref{uniquetrans}.]

Let $B$ be a normal abelian subgroup of $G$ such that $G/B$ is cyclic of order $d$. Necessarily, the commutator subgroup $[G,G]$ must be contained in $B$. 

\underline{$d=6$}: by Proposition \ref{wallgrps} $[G,G]$ has  index six in $G$, this implies $B =[G,G] = A$. 

\underline{$d=4$}: here $G$ is a quotient of the triangle group  $\mathbb T(2,4,4)$  (cf. Remark \ref{tworems}) which has three normal subgroups of index four, one is abelian and isomorphic to $\mathbb Z^2$, the other two groups are non abelian with abelianisation $\mathbb Z_2^3$. 
If $|G|>16$, then $|B| >4$. The preimage of $B$ under $\phi \colon \mathbb T(2,4,4) \to G$ is a normal subgroup of index four and since $B$ is abelian, there is a surjection $\phi^{ab} \colon \phi^{-1}(B)^{ab} \to B$.  Assume now that  $\phi^{-1}(B)^{ab} \simeq \mathbb Z_2^3$. Then  $\phi^{ab}$ is an isomorphism from $\mathbb Z_2^3$ to $B$. A contradiction, since we assume that 
$B$ is a subgroup of $E[n] \simeq \ZZ_n^2$.  Thus $\phi^{-1}(B)^{ab} = \phi^{-1}(B) \simeq \mathbb Z^2$ and $B=A$.
Suppose   that $|G| \leq 16$ and $B \neq A$, then $B$ is a quotient of $\mathbb Z_2^2$ and  $G$  is equal to 
$\mathbb Z_2 \times \mathbb Z_4$ or  $\mathbb Z_2^2 \rtimes_{\varphi_4} \mathbb Z_4$.
%We can exclude $\mathbb Z_4$ and  $\mathbb Z_2 \times \mathbb Z_4$, because  they have a unique subgroup with a cyclic quotient of order four. 
%To exclude $G=\mathbb Z_2^2 \rtimes_{\varphi_4}   \mathbb Z_4$, we consider the 
%surjective homomorphism 
%\[
%\beta \colon \mathbb Z_2^2 \rtimes_{\varphi_4}  \mathbb Z_4  \to \mathbb Z_2 \times \mathbb Z_4, \qquad (a,b,c) \mapsto (a+b,c).  
%\] 
%Its kernel is the commutator subgroup $[G,G]$. Thus $ker(\beta)=[G,G]\subset B$ and $B$ must be  the preimage of 
%a normal subgroup of $\mathbb Z_2 \times \mathbb Z_4$ with cyclic quotient of order four. 
%There is only one such subgroup  namely $\langle (1,0) \rangle \leq  \mathbb Z_2 \times \mathbb Z_4$.
%This implies  $B=\beta^{-1}(\langle (1,0) \rangle)=A$, contradiction. 

\underline{$d=3$}: if $|G| >27$, then $|B| >9$. The preimage of $B$ under $\phi \colon \mathbb T(3,3,3) \to G$ is a normal subgroup of index three. Since $B$ is abelian, there is a surjection 
$\phi^{-1}(B)^{ab} \to B$. It is known that $\mathbb T(3,3,3)$ 
 has exactly  four normal subgroups of index three. One of them is isomorphic to  $\mathbb Z^2$ and the others are non-abelian, with abelianisation 
$\mathbb Z_3^2$. It follows that $\phi^{-1}(B)^{ab} = \phi^{-1}(B) \simeq \mathbb Z^2$ and $B=A$. Suppose that $|G| \leq 27$ and $B \neq A$, then $B$ is a quotient of $\mathbb Z_3^2$. 
We conclude that $G$ has order $9$ or $27$ and  is equal to 
$\mathbb Z_3^2$ or  $\He(3)=\mathbb Z_3^2 \rtimes_{\varphi_3} \mathbb Z_3$.
\end{proof}

\begin{corollary}\label{nonexfixed}
Assume that $G$ is not exceptional and admits a rigid diagonal action on $E^n$.  Then: 
\begin{enumerate}
\item
the action is not free;
\item 
the maps $g(z_1), \ldots, g(z_n)$ have the same linear part for all $g \in G$.
\end{enumerate}
\end{corollary}

\begin{proof}
We know that $G = A \rtimes_{\varphi_d} \mathbb Z_d$ with $d=3,4$ or $6$.

(1) By Proposition \ref{uniquetrans} we have $T_i = A$ for all $i$. The claim follows now from Remark \ref{transgroup} (2). 

(2) We denote the curve 
$E$ at position $i$ by $E_i$. 
According to Remark \ref{canrep} the value of the character $\chi_{E_i}(g)$ is the complex conjugate of the linear part of $g(z_i)$.  
Proposition \ref{uniquetrans} implies that $\ker(\chi_{E_i}) =A$,  i.e.  we can 
regard $\chi_{E_i}$ as a faithful character of $G/A \simeq \mathbb Z_d$. Such a character is defined by multiplication with a $d$-th primitive root of unity. 
For $d=3,4$ and $6$ there are only two primitive roots of unity, namely
$\zeta_d$ and $\zeta_d^{-1}$. Since we assume that the action is rigid, by  Remark  \ref{canrep}   it holds that
$\chi_{E_i} \cdot \chi_{E_j}\neq \chi_{triv}$ and  it follows that  
 all characters $\chi_{E_i}$ are the same. 
\end{proof}
 The following result shows that for the non exceptional case the situation is quite simple, since for each $n$ and each $d$ there is only one possible quotient, namely $X_{n,d}:=E^n/\mathbb Z_d$.
\begin{theorem}\label{oneisoclass}
Assume that $G=A \rtimes_{\varphi_d} \mathbb Z_d$ admits a rigid  diagonal action on $E^n$. If $G$
is not exceptional, then the quotient  
$E^n/G$ is isomorphic to $X_{n,d}:=E^n/\mathbb Z_d$, where $\mathbb Z_d$ acts on $E^n$ by multiplication with $\zeta_d \cdot \Id$. 
\end{theorem}

\begin{proof}
Since $G$ is not exceptional, the normal subgroup $A$  is  the translation group $T_i$ of the action on each factor of the product $E^n$. 
In particular   $E^n/A$ is an abelian variety and by construction $\mathbb Z_d \simeq G/A$ acts on $E^n/A$ by multiplication with  $\zeta_d \cdot \Id$. This implies that $E^n/A$ is isomorphic to $E^n$, according  to \cite[Corollary 13.3.5]{BL}. 
\end{proof}

\begin{rem}\label{Beau}
In \cite{beauville} Beauville considered the varieties $X_{d,d}$ for $d=3,4$ and $6$. He explains that these singular varieties 
admit rigid, simply connected Calabi-Yau $d$-folds as resolutions of singularities. 
\end{rem}

\section{The exceptional groups}

In contrast to Theorem \ref{oneisoclass} the situation for the exceptional groups is more involved.  In this chapter we  analyse it in dimensions three and four. 
In Proposition  \ref{norigid}  we shall see that
 rigid, free diagonal actions do not exist in dimension three. 
 However we can drop the freeness assumption and classify rigid non-free actions. It turns out that the quotients by the groups 
 $\mathbb Z_2 \times \mathbb Z_4$ and $\mathbb Z_2^2 \rtimes_{\varphi_4} \mathbb Z_4$ have non canonical singularities and therefore we do not consider it further.

 In case of the groups $\He(3)$ and $\mathbb Z_3^2$, the  quotients will be singular, but as we shall see only canonical cyclic quotient singularities of type 
$ \frac{1}{3}(1,1,1)$ and $\frac{1}{3}(1,1,2)$

 We study this case in detail and discover 
 besides Beauville's example $X_{3,3}$,  other interesting rigid canonical threefolds and relations among them.  
From dimension four on, the groups $\He(3)$ and $\mathbb Z_3^2$ allow rigid free  diagonal actions.  
Their existence for $\mathbb Z_3^2$  has already been observed  in \cite[Theorem 3.4]{rigidity}. We show that  each exceptional  group gives  precisely one isomorphism class of a smooth rigid quotient  fourfold $E^4/\mathbb Z_3^2$  and 
 $E^4/\He(3)$. 
We prove that these manifolds  are non-isomorphic by showing that they are even  topologically distinct, i.e. non homeomorphic. 

\begin{rem}
If we consider rigid actions of the exceptional groups  $\mathbb Z_2 \times \mathbb Z_4$ and $\mathbb Z_2^2 \rtimes_{\varphi_4} \mathbb Z_4$ on $E^4$, then as observed before the action cannot be free, hence we obtain  singular quotients, which have canonical singularities. Classifying these quotients as well as suitable resolutions of singularities of those should yield new interesting examples.
\end{rem}
 
 \subsection{Free rigid diagonal actions on $E^4$} \ 
 
We start with the following quite easy but useful Proposition
%criterion for a free action of the exceptional groups on a product of elliptic curves. 

\begin{proposition}\label{norigid} \
Assume that $G=\He(3)$ or $\mathbb Z_3^2$  admits a rigid diagonal action on a product $E^n$. Then the action is free, if and only if 
each of the four normal subgroups of $G$ of index three is equal to  the translation subgroup $T_i$ on at least one factor.
In particular,  $n\geq 4$. 
\end{proposition}

\begin{proof}
By Remark \ref{transgroup} (2),  the action is free if and only if
\begin{equation}\label{unionTi}
G=T_1 \cup \ldots \cup T_n.  
\end{equation}
Recall that each $T_i$ is equal to one of the four normal subgroups $A_1, \ldots, A_4$ of $G$ of index three. 
Clearly,  for (\ref{unionTi}) to hold,  all $A_i$ must appear in  the union,  because 
the union of less than four of the $A_i$'s  
consists of at most $7$ elements if  $G=\mathbb Z_3^2$ and of at most $21$ elements if $G=\He(3)$.

\end{proof}

We need the following Proposition only in a special case but for further use we prefer to state it in greater generality:
\begin{proposition}\label{coverf}
Let $G$ be a finite group acting holomorphically on  the compact complex manifolds $X_1, \ldots, X_n$,  let $A\trianglelefteq G$ be a normal subgroup and $Y_i:=X_i/A$.

1) There is a commutative square of finite holomorphic maps:
$$
\xymatrix{
X_1\times \ldots \times X_n  \ar[r]  \ar[d] & Y_1 \times \ldots \times Y_n, \ar[d] \\
(X_1\times \ldots \times X_n)/G \ar[r]_{f}&  (Y_1\times \ldots \times Y_n)/(G/A), 
}
$$
where the degree of $f$ is $\vert A \vert^{n-1}$.

2) If $n\geq 2$ the covering $f$  is Galois if and only if $A$ is contained in the center of $G$.  In this case the Galois group of $f$ is isomorphic to 
$A^{n-1}$. 

\end{proposition}

\begin{proof} 
(1) is obvious.

(2)  Consider the composition of covers 
\[
h \colon X_1\times \ldots \times X_n \to Y_1 \times \ldots \times Y_n \to (Y_1\times \ldots \times Y_n)/(G/A).
\]

We claim that $h$ is always Galois. Observe that the elements of the groups  $\Delta_G$ and $A^n$ act as  deck transformations of $h$ and therefore 
in order to show that $h$ is Galois, it suffices to show that the cardinality of  $\langle \Delta_G, A^n \rangle$ equals $\deg (h) = |A|^{n-1}|G|$.  Since $A$ is a normal subgroup of $G$, also $A^n$ is a normal subgroup of $\langle \Delta_G, A^n \rangle$. Clearly the natural homomorphism 
\[
\Delta_G \to \langle \Delta_G, A ^n \rangle/A^n
\]
is surjective and its kernel is $\Delta_A \simeq A$.  Thus we see that  $\langle \Delta_G, A^n \rangle$ has $|A|^{n-1}|G|$  elements.

 By the fundamental theorem of Galois theory, $f$ is Galois if and only if $\Delta_G$ is a normal subgroup of $\langle \Delta_G, A^n\rangle$. 
This  is certainly the case if   $A\leq Z(G)$. Conversely, assume that $\Delta_G$ is normal. Let $a\in A$, then for all $g\in G$ there exists an element $g'\in G$, such that 
\[
(a,1,\ldots,1 ) \circ (g, \ldots,g) \circ (a^{-1},1,\ldots,1) =(g',\ldots,g'). 
\]
This implies $aga^{-1}=g'=g$ i.e. $a \in Z(G)$. 

Assume now that   $f$ is Galois. Then its Galois group is $\langle \Delta_G, A^n \rangle/\Delta_G$, which is isomorphic to  $A^{n-1}$. In fact, 
the  surjective homomorphism
\[
A^n \to \langle \Delta_G, A^n \rangle/\Delta_G
\]
has kernel   $\Delta_A$ and induces an  isomorphism $\langle \Delta_G, A^n \rangle/\Delta_G  \simeq A^n/\Delta_A \simeq A^{n-1}$. 
\end{proof}
Now we are ready to show that for each of the groups  $\He(3)$ and $\mathbb Z_3^2$,  there is exactly one isomorphism class of rigid \'etale quotients $E^4 /G$.

\begin{theorem}\label{Z1Z2Glatt}
There is exactly one 
isomorphism class of quotient manifolds 
$Z_1:=E^4/\mathbb Z_3^2$ resp. $Z_2:=E^4/\He(3)$ obtained by a rigid free and diagonal action.  
These projective  manifolds are infinitesimally rigid of Kodaira dimension zero  and there is an  
unramified Galois  cover  $f \colon Z_2 \to Z_1$ with  group $\mathbb Z_3^3$. 
\end{theorem}

\begin{proof}
First we 
show the existence and unicity of $Z_1$ resp. $Z_2$. 
Let $G$ be $\mathbb Z_3^2$ resp. $\He(3)$. By
Riemann's existence theorem the diagonal actions of $G$ on $E^4$ correspond to quadruples of generating triples
$[V_1,V_2,V_3,V_4]$.
The action of $G$ is free on $E^4$, if and only if  
$\Sigma_{V_1} \cap \ldots \cap \Sigma_{V_4} = \lbrace 1_G \rbrace$. 

As explained e.g. in \cite{BeauReal} the group $\mathfrak S_4 \times \mathcal B_3^4 \times \Aut(G)$ acts on the set of quadruples of generating triples $[V_1,V_2,V_3,V_4]$.
Here  $\mathfrak S_4$ permutes the generating triples $V_i$ of the quadruple, $\Aut(G)$ acts diagonally on 
$[V_1,V_2,V_3,V_4]$, and the {\em Artin Braid Group} $\mathcal B_3$ acts separately on each $V_i$ by so-called {\em Hurwitz moves}. By  \cite[Proposition 3.3]{BeauReal} equivalent quadruples of generating vectors yield isomorphic quotients.
With a MAGMA algorithm  we check that for each of the two groups there is exactly one orbit corresponding to a free rigid action, corresponding 
therefore to unique isomorphism class of a   rigid manifold  $Z_1:=E^4/\mathbb Z_3^2$ resp. $Z_2:=E^4/ \He(3)$.

Let us  now consider the $\He(3)$ action on $E^4$ yielding $Z_2$.
The center $C_3=Z\big(\He(3)\big) \simeq \mathbb Z_3$ 
acts on each copy of $E$ by translations, such that 
$E/C_3\simeq E$.  
Using the identification $\He(3)/C_3 \simeq \mathbb Z_3^2$, it can be checked again by a MAGMA routine that the image of the quadruple $[V_1,V_2,V_3,V_4]$ representing $Z_2$ is a quadruple which lies in the orbit representing $Z_1$. This means that on each factor  we have a commutative triangle 
$$
\xymatrix{
E_i  \ar[r]  \ar[d]_{\He(3)} & E_i/C_3 \ar[dl]^{\mathbb Z_3^2} \\
\mathbb P^1
}
$$
By
Proposition \ref{coverf} we have  a finite Galois cover 
$f \colon Z_2 \to Z_1$
 with group $\mathbb Z_3^3$. The cover $f$ is unramified, because the other three maps 
 of the diagram in Proposition \ref{coverf} are unramified. 
\end{proof}

\begin{rem}\label{motivConstr}
Note that the fourfold  $Z_2$ can be realized as a double quotient, namely  the quotient of the torus $E^4/C_3$
 by the induced  action of 
$\mathbb Z_3^2 \simeq \He(3)/C_3$.  
By construction, the linear part of the $\mathbb Z_3^2$-action on $E^4/C_3$ giving 
$Z_2$
is the same as the 
linear part of the action on $E^4$ giving  $Z_1$. 
It can be determined from the generating triples, see Remark \ref{RemRET} (3). In our situation we have (up to an automorphism of $\mathbb Z_3^2$): 
\[
\rho(a,b):=
\begin{pmatrix} \zeta_3^b  &  0 & 0 & 0  \\ 0 & \zeta_3^{a+b} & 0 & 0 \\ 0 & 0 & \zeta_3^a & 0  \\ 0 & 0 & 0 & \zeta_3^{2a+b} \end{pmatrix}, \qquad \makebox{for all} \quad  (a,b) \in \mathbb Z_3^2. 
\]
In the literature $\rho$ is usually called the \emph{analytic representation} of the group action. 
\end{rem}
With the above  in mind, it is not hard to write down explicit models of $Z_1$ and $Z_2$:  
\begin{example}\label{explicitex}
We consider the following  lattices 
\[
\Lambda_1:=\mathbb Z[\zeta_3]^4 \qquad \makebox{and} \qquad \Lambda_2 = \mathbb Z[\zeta_3]^4 + \mathbb Z \frac{1+2\zeta_3}{3}  (1,1,1,1)
\]
and define the complex tori $T_i:=\mathbb C^4/\Lambda_i$. By definition 
 $T_1 \simeq E^4$ and $T_2 \simeq E^4/\langle t \rangle$, where $E$ is the equianharmonic elliptic curve and 
\[
t \colon E^4 \to E^4, \qquad  z \mapsto z+  \frac{1+2\zeta_3}{3}  (1,1,1,1). 
\]
The  two actions $\psi_i$ of $\mathbb Z_3^2$ on $T_i$,  that give  the quotients $Z_i$ are  the following: 
\begin{align*}
\psi_1(1,0)(z) & := \diag(1, \zeta_3, \zeta_3, \zeta_3^2)z  + \frac{1+2\zeta_3}{3} (1, 2,0,1), \\
\psi_1(0,1)(z) & :=\diag (\zeta_3, \zeta_3, 1, \zeta_3) z +  \frac{1+2\zeta_3}{3} (0,0,2,0), \\
\psi_2(1,0)(z)  & :=\diag (1, \zeta_3, \zeta_3, \zeta_3^2)z  + \frac{1}{3} (1, 0 , 0, 2), \\
\psi_2(0,1)(z)  & :=\diag(\zeta_3, \zeta_3, 1, \zeta_3) z + \frac{1}{3}(0, 2 \zeta_3, 1+ \zeta_3, 2).
\end{align*}
\end{example}

The cohomology of the manifolds $Z_i$ is easy to compute:

\begin{proposition}\label{Hodgesmooth}
The projective manifolds $Z_1$ and $Z_2$ have the same Hodge numbers:
\[
h^{1,0}=h^{2,0}=h^{3,1}=h^{4,0}=0, \quad h^{1,1}=4, \quad  h^{2,1}=3,\quad h^{3,0}=1, \quad  h^{2,2}=6. 
\]
Moreover, $\hol(K_{Z_i}) \neq \hol_{Z_i}$ and   $\hol(K_{Z_i}) ^{\otimes 3} \simeq \hol_{Z_i}$.  
\end{proposition}

\begin{proof}
For any complex torus $T=\mathbb C^n/\Lambda$ the Dolbeault groups have the description 
\[
 H^{p,q}(T)= \Lambda^p \Omega \otimes \Lambda^q \overline{\Omega}, \quad \makebox{where} \quad \Omega:=\langle dz_1, \ldots, dz_n \rangle.  
\]
Since the manifolds $Z_i$ are quotients of tori by free actions,  the groups $H^{p,q}(Z_i)$
are isomorphic to the invariant 
parts of $\Lambda^p \Omega \otimes \Lambda^q \overline{\Omega}$ under  the $\mathbb Z_3^2$ action induced by  $\psi_i$. 
Since the derivative of a constant is zero, it suffices to act with the linear part of $\psi_i$ i.e. with the analytic representation. 
Since both actions $\psi_i$ have the same analytic representation $\rho$,  both quotients   $Z_1$ and $Z_2$ have isomorphic Dolbeault groups and in particular, the 
same Hodge numbers.  To compute these groups explicitly we take the standard  basis 
\[
\mathcal B:=\lbrace dz_{i_1} \wedge \ldots \wedge dz_{i_p} \otimes d\overline{z}_{j_1} \wedge \ldots \wedge d\overline{z}_{j_q} ~ \big\vert ~   i_1 < \ldots < i_p \leq 4, ~  j_1 < \ldots < j_q \leq 4  \rbrace 
\]
of  $\Lambda^p \Omega \otimes \Lambda^q \overline{\Omega}$. 
The fact that $\rho$ acts by diagonal matrices implies that 
a basis of  $H^{p,q}(Z_i)$ is given by the invariant basis vectors of  $\mathcal B$. 
The non-zero Dolbeault groups are: 
\begin{align*}
H^{3,0}(Z_i)& \simeq  \langle dz_1\wedge dz_2 \wedge dz_4 \rangle, \\
H^{1,1}(Z_i)& \simeq  \langle dz_i\otimes d\overline{z}_i ~ \big\vert ~ i \leq 4 \rangle, \\
H^{2,1}(Z_i)& \simeq  \langle dz_1 \wedge dz_3 \otimes d\overline{z}_2, dz_2 \wedge dz_3 \otimes d\overline{z}_4, dz_3 \wedge dz_4 \otimes d\overline{z}_1 \rangle, \\
H^{2,2}(Z_i) & \simeq \langle dz_i \wedge dz_j \otimes d\overline{z}_i \wedge d\overline{z}_j ~ \big\vert ~ i < j \leq 4 \rangle.
\end{align*}
To prove the statement about  $\hol(K_{Z_i})$, we note  that the differential form
$$\big( dz_1 \wedge \ldots \wedge dz_4)^{\otimes 3}$$ is $\mathbb Z_3^2$-invariant. Thus  it 
descends to $Z_i$ and provides a trivialization of   $\hol(K_{Z_i})^{\otimes 3}$. 

\end{proof}

The remaining part of the subsection is devoted to prove that  the manifolds $Z_1$ and $Z_2$ are not homeomorphic. More precisely, we shall show that they have
  non isomorphic fundamental groups. 

\begin{rem}
The fundamental group  of $Z_i$  is isomorphic to  
the group of deck  transformations $\Gamma_i$ of the universal cover 
$\mathbb C^4 \to T_i \to Z_i$. 
It  consists of the lifts of the automorphisms $\psi_i(a,b)$ for all $(a,b) \in \mathbb Z_3^2$ and is therefore a group of affine transformations. 
Since the linear parts $\rho(a,b)$ of the maps  $\psi_i(a,b)$, viewed as  real $8\times 8$ matrices, are orthogonal  we can more precisely say that 
$\Gamma_i$ is a cocompact free
discrete subgroup of of the Euclidean group of isometries 
$\mathbb{E}(8):= \mathbb R^8 \rtimes \OO(8)$.

Because  the action of $\mathbb Z_3^2$ on $T_i$ does not contain translations, the lattice $\Lambda_i$ of the torus $T_i$ is equal to the intersection  $\Gamma_i \cap \mathbb R^8$ i.e. the translation subgroup of $\Gamma_i$. 
\end{rem}

\begin{definition}
1) A discrete cocompact subgroup of $\mathbb{E}(n)$ is called a {\em crystallographic group}.  

2) A {\em Bieberbach group} is a torsion free crystallographic group. 
\end{definition}

As a modern  reference for Bieberbach groups we use \cite{LCh}, for the original results see \cite{bib1}, \cite{bib2}.

\begin{rem}
1) It is worth observing that the underlying $\mathcal C^{\infty}$-manifold of $Z_i$ admits a \emph{flat Riemannian metric}  i.e., a metric such that  the curvature tensor 
\[
R(X,Y)Z:=\nabla_X \nabla_YZ - \nabla_Y \nabla_XZ - \nabla_{[X,Y]}Z
\]
with respect to the \emph{Levi-Civita connection} is identically zero. Vice versa each compact flat Riemannian $n$-manifold is isometric to a  quotient $\RR^n /\Gamma$, where $\Gamma$ is a Bieberbach group (cf. \cite[Chapter II]{LCh}). Moreover, the quotient $\Gamma / \Lambda$ by the translation subgroup  is isomorphic to the holonomy group of $\RR^n /\Gamma$.

2) Obviously not every quotient of $\RR^{2n}$ by a Bieberbach group has a complex structure. If there is a complex structure, then the complex manifold is an \'etale torus quotient and is called a {\em generalized hyperelliptic manifold}. These have been studied and classified in dimension 2 by Bagnera and de Franchis and in dimension 3 by Uchida-Yoshihara  \cite{UchidaYoshi}, Lange \cite{Lange} and Catanese-Demleitner in \cite{AndiFab}. 
In his PhD thesis \cite{Demleitner} Demleitner gave a complete list of holonomy groups of generalized hyperelliptic $4$-folds.  The manifolds 
$Z_1$ and $Z_2$ are two distinct {\em rigid} examples,  with holonomy group $\mathbb Z_3^2$. 
The second author and Demleitner work on a complete classification of rigid generalized hyperelliptic $4$-folds. 
\end{rem}

In order to distinguish the fundamental groups  $\Gamma_1$ and $\Gamma_2$ of $Z_1$ and $Z_2$ we will use the first and second of the following three  theorems of Bieberbach 
(cf. \cite[Chapter I]{LCh}):

\begin{theorem}[Bieberbach's three theorems]\label{biberer} \

(1) The translation subgroup $\Lambda:=\Gamma \cap \mathbb R^n$ of a crystallographic group  $\Gamma \leq \mathbb{E}(n)$ is a 
lattice  of rank $n$ and $\Gamma/\Lambda$ is finite.  All  other normal abelian subgroups of $\Gamma$ are contained in $\Lambda$.

(2) Let  $\Gamma_1, \Gamma_2 \leq \mathbb{E}(n)$ be two crystallographic groups and $f \colon \Gamma_1 \to \Gamma_2$ be an isomorphism. Then there  exists 
an affine transformation  
 $\alpha \in \Aff(n)$, such that $f(g)=\alpha \circ g \circ \alpha^{-1}$ for all $g\in \Gamma_1$.  

(3) In each dimension  there are only finitely many isomorphism classes of crystallographic groups. 
\end{theorem}

\begin{rem}\label{diffeovarphi} 
1) Assume that $Z_1$ and $Z_2$ are homeomorphic. Then, by Bieberbach's  second theorem,  there exists an
affine transformation $\alpha(x)=Ax+b$ such that $\alpha \circ \Gamma_2 \circ \alpha^{-1}=\Gamma_1$.  Bieberbach's  first  theorem implies that
$\alpha \circ \Lambda_2 \circ \alpha^{-1}=\Lambda_1$. In other words, $\alpha$ induces  diffeomorphisms
\[
\widehat{\alpha} \colon Z_2 \to Z_1 \qquad \makebox{and} \qquad  \widetilde{\alpha} \colon \mathbb R^8/\Lambda_2 \to \mathbb R^8/\Lambda_1
\]
which  make the following diagram commutative:  
 \[
\xymatrix{
\mathbb R^8  \ar[r]^{\alpha} \ar[d] & \ar[d] \mathbb R^8 \\
\mathbb R^8/\Lambda_2  \ar[d]\ar[r]^{\widetilde{\alpha}} & \mathbb R^8/\Lambda_1  \ar[d] \\
Z_2\ar[r]^{\widehat{\alpha}} & Z_1.} 
\]
In particular, Bieberbach's  theorems imply  that $Z_1$ and $Z_2$ have isomorphic fundamental groups if and only if they are diffeomorphic, even by an affine diffeomorphism.

2) Recall that the actions $\psi_1$ and $\psi_2$ have the same analytic representation $\rho$, which we now view  
as a real representation:
\[
 \rho_{\mathbb R}(a,b) := 
\begin{pmatrix} B^b  &  0 & 0 & 0  \\ 0 & B^{a+b} & 0 & 0 \\ 0 & 0 & B^a & 0  \\ 0 & 0 & 0 & B^{2a+b}
\end{pmatrix},
\quad B:=-\frac{1}{2}
\begin{pmatrix} 1 &  \sqrt{3}  \\ -\sqrt{3} &1 \end{pmatrix}. 
\]
By the commutativity of the diagram in (1), there exists an automorphism 
$\varphi \in \Aut(\mathbb Z_3^2)$,  such that 
$$
A \rho_{\mathbb R}(a,b) A^{-1} =\rho_{\mathbb R}\big(\varphi(a,b)\big),  \ \forall \  (a,b) \in \mathbb Z_3^2. 
$$
In other words, $\rho_{\mathbb R}(a,b) $ and $\rho_{\mathbb R}\big(\varphi(a,b)\big)$ 
are isomorphic as representations over $\mathbb R$. 

\end{rem}

\begin{proposition}\label{4twodims}
The representation $\rho_{\mathbb R}$ is the sum of four disctinct irreducible two-dimensional representations over $\mathbb R$. 
\end{proposition}

\begin{proof}
The representations $B^b,  B^{a+b},  B^a$ and  $B^{2a+b}$  are indeed irreducible, because $B$ is not diagonalizable over $\mathbb R$. 
Obviously they are distinct.
\end{proof}

\begin{rem} 
1) The group $\mathbb Z_3^2$ has precisely $5$ irreducible real representations: the trivial representation and the four two dimensional representations from above. 
This can be verified with the help of the formula 
\[
|G| = \sum_{\chi \in \Irr_{\mathbb R}(G)} \frac{\chi(1)^2}{\langle \chi, \chi \rangle} ,
\]
which  holds for any finite group  $G$. 
% Note that  the sum is taken over the  set of characters of the irreducible real representations of $G$.

2) By Schur's Lemma the  endomorphism algebra $\End_G(V) $  of an  irreducible real representation $V$ of a finite group $G$  is a finite dimensional  division algebra.  
As it is clearly associative, it  is isomorphic to 
$\mathbb R$, $\mathbb C$ or the quaternions $\mathbb H$, according  to Frobenius' theorem \cite{frob}.   

\end{rem}

\begin{proposition}\label{comMatrices} \

1) The $\mathbb R$-algebra of matrices $H$ which commute with $B$ is:  
\[
\bigg\lbrace 
\begin{pmatrix} 
\lambda  &  -\mu  \\ \mu & \lambda
\end{pmatrix} ~ \bigg\vert ~ \lambda,\mu \in \mathbb R \bigg\rbrace
 \simeq \mathbb C. 
\]

2) The $\mathbb R$-vectorspace of matrices $H$ with  $HB=B^2H$ is 
\[
\bigg\lbrace 
\begin{pmatrix} 
\lambda  &  \mu  \\ \mu & -\lambda
\end{pmatrix} ~ \bigg\vert ~ \lambda,\mu \in \mathbb R \bigg\rbrace
  \simeq \mathbb R^2. 
\]

The matrices in $1)$ define $\mathbb C$-linear maps and the matrices in $2)$ $\mathbb C$-antilinear maps. 
In complex coordinates $z=x+ i y$ we may  identify them with  
\[
h_{\lambda + i  \mu} \colon \mathbb C \to \mathbb C, \quad z\mapsto (\lambda+i\mu)z \qquad \makebox{and} \qquad 
\overline{h}_{\lambda + i \mu} \colon \mathbb C \to \mathbb C, \quad z\mapsto (\lambda+i\mu)\overline{z}. 
\]
\end{proposition}

\begin{proof}
An Ansatz with a general 
$2\times 2$ matrix $H$ yields a  system of linear equations. The solutions are the displayed matrices.  
\end{proof}
Now we are ready to prove the following:
\begin{theorem}\label{Z1notZ2}
The fundamental groups of the manifolds $Z_1$ and $Z_2$ are not isomorphic. 
\end{theorem}

\begin{proof}
Assume  the converse. Then, as we have seen in Remark \ref{diffeovarphi}, 
 there exists an affine transformation $\alpha(x)=Ax+b$ inducing a diffeomorphism 
$\widehat{\alpha} \colon Z_2 \to Z_1$ and an  automorphism
$\varphi \in \Aut(\mathbb Z_3^2)$, such that 
\[
A \rho_{\mathbb R}(a,b) A^{-1} =\rho_{\mathbb R}\big(\varphi(a,b)\big),  \qquad \makebox{for all}  \qquad  (a,b) \in \mathbb Z_3^2. 
\]
We subdivide $A$ into  $16$ blocks $A_{i,j}$ of $2\times 2$-matrices:
\[
A=
\begin{pmatrix}
A_{1,1} & A_{1,2} & A_{1,3} & A_{1,4} \\
A_{2,1} & A_{2,2} & A_{2,3} & A_{2,4} \\
A_{3,1} & A_{3,2} & A_{3,3} & A_{3,4} \\
A_{4,1} & A_{4,2} & A_{4,3} & A_{4,4} 
\end{pmatrix}
\]
By Proposition \ref{4twodims} the representation   $ \rho_{\mathbb R}$, and henceforth also $\rho_{\mathbb R}\circ \varphi$,  is made up of four distinct irreducible 
two-dimensional  representations of $\mathbb Z_3^2$. 
 Schur's lemma tells us  that  there exists a 
permutation $\tau \in \mathfrak S_4$ such that  for each $i$ the block 
$A_{\tau(i),i}$ is invertible, while the other $12$  blocks  are identically  zero. The non-zero blocks are those described in  Proposition \ref{comMatrices}:  each block 
$A_{\tau(i),i}$ either commutes with $B$, or $A_{\tau(i),i}  B=B^2A_{\tau(i),i}$.  
Whence up to a permutation of blocks,  $A$ is a sum of $\mathbb C$-linear and  $\mathbb C$-antilinear maps:  
\[
h_{w_i}(z)=w_i z \qquad \makebox{or} \qquad 
\overline{h}_{w_i}(z)=w_i \overline{z} \qquad \makebox{where} \quad w_i  \in \mathbb C^{\ast}. 
\]
Since $\alpha$ defines at the same time a  diffeomorphism between the tori  $T_2 $ and $T_1$, it holds
\[
A \cdot \Lambda_2 = \Lambda_1 =\mathbb Z[\zeta_3]^4,
\] 
where we view $A$ as a map $A\colon \mathbb C^4 \to \mathbb C^4$. 
In particular $Ae_i \in\mathbb Z[\zeta_3]^4$, for all $1\leq i \leq 4$. This shows that all $w_i$ belong to  $ \mathbb Z[\zeta_3]$. 
Similarly, since  $A^{-1}e_i$ is a vector with just one non-zero entry, it must be  contained in the sublattice 
\[
 \mathbb Z[\zeta_3]^4 \subset \Lambda_2= \mathbb Z[\zeta_3]^4 + \mathbb Z \frac{1+2\zeta_3}{3}  (1,1,1,1). 
\]
Thus,  
$w_{i}^{-1}$ or its conjugate is also an Eisenstein integer and  we conclude that  $w_i$ is a unit in $ \mathbb Z[\zeta_3]$, for all $1\leq i \leq 4$.   
On the other hand,  the  product of $A$ and the lattice vector $\frac{1+2\zeta_3}{3}  (1,1,1,1) \in \Lambda_2$ belongs  to $\Lambda_1$, which  means that 
\[
w_i \frac{1+2\zeta_3}{3} \in \mathbb Z[\zeta_3] \qquad \makebox{or} \qquad 
 w_i \frac{1+2\zeta_3^2}{3} \in \mathbb Z[\zeta_3]. 
\]
A contradiction. 
\end{proof}

\subsection{Rigid quotients of $E^3$ by the exceptional groups}

Let $X := E^3 / G$ be a quotient of  $E^3$ by a rigid diagonal action of one of the four exceptional groups $G$. Then   according to Proposition \ref{norigid} and Remark \ref{structureHeis}, (2)  the action is not free and $X$ has singular points. 
The singular points of  $X$ are precisely the images of the finitely many   points $p=(p_1,p_2,p_3)$ in $E^3$ with non trivial stabilizer group. 
The stabilizer of a point $p\in E^3$ is the intersection of the cyclic groups $G_{p_i}$ and therefore cyclic (cf. Remark \ref{RemRET}).  
We show that if $G = \mathbb Z_2 \times \mathbb Z_4$ or $\mathbb Z_2^2 \rtimes_{\varphi_4}  \mathbb Z_4$, then $X$ always has non  canonical singularities of type $\frac{1}{4}(1,1,1)$. Therefore we shall restrict ourselves to the two remaining exceptional groups $\mathbb Z_3^2$ or $\He(3)$.
Since a  non trivial cyclic subgroup of $\mathbb Z_3^2$ or $\He(3)$ is isomorphic to $\mathbb Z_3$, the threefold  $X$ has only cyclic quotient singularities of type 
$\frac{1}{3}(1,1,1)$ and $\frac{1}{3}(1,1,2)$. 
We point out that these singularities are canonical, more precisely 
$\frac{1}{3}(1,1,2)$ is terminal by the  Shepherd-Barron-Tai criterion (see \cite[ p. 376 Theorem]{R87})  and  $\frac{1}{3}(1,1,1)$ is Gorenstein. 
In particular, for any resolution of singularities $\rho \colon \widehat{X} \to X$ it holds $\kappa(\widehat{X})=\kappa(X)= \kappa(E^3)=0$. 

In terms   of infinitesimal deformation theory these singularities  are also well behaved in the following sense:

\begin{proposition}\label{resolution}
Let $X$ be a threefold with only isolated singularities of type $\frac{1}{3}(1,1,1)$ and  $\frac{1}{3}(1,1,2)$, then $X$ is canonical and there is  a resolution of singularities 
$\rho \colon \widehat{X} \to X$, such that   $H^1(\widehat{X}, \Theta_{\widehat{X}}) \simeq H^1(X, \Theta_X)$. 

In particular, if $X$ is rigid, then also $\widehat{X}$ is rigid.
\end{proposition}

\begin{proof}
In \cite[Corollary 5.9, Proposition 5.10]{BG} the authors showed that a germ $(U,p_0)$ of a singularity of type $\frac{1}{3}(1,1,1)$ or $\frac{1}{3}(1,1,2)$ has a resolution $\rho \colon \widehat{U} \to U$ such that 
\[
\rho_{\ast} \Theta_{\widehat{U}}= \Theta_U \qquad \makebox{and} \qquad R^1  \rho_{\ast} \Theta_{\widehat{U}}=0.  \qquad \qquad (\ast)
\]
These resolutions can be glued to obtain a resolution $\rho \colon \widehat{X} \to X$ with the same property $(\ast)$. The low term exact sequence of the Leray spectral sequence 
\[
0 \to H^1(X,\rho_{\ast} \Theta_{\widehat{X}}) \to H^1(\widehat{X},\Theta_{\widehat{X}}) \to H^0(X,R^1\rho_{\ast} \Theta_{\widehat{X}}) \to \ldots 
\]
gives us an  isomorphism $H^1(\widehat{X},\Theta_{\widehat{X}})  \simeq H^1(X,\Theta_X)$. 
\end{proof}

\begin{rem}
Since canonical singularities are rational (see e.g.  \cite[ p. 363 (3.8)]{R87}), Leray's spectral sequence implies that for any resolution of singularities $\rho \colon \widehat{X} \ra X$
the irregularities 
\[
q_i(\widehat{X}) := h^i(\widehat{X},\hol_{\widehat{X}}) \qquad \makebox{and} \qquad  q_i(X):= h^i(X, \hol_X)
\]
 coincide. Since  $H^{i,0}(\widehat{X})  \simeq H^{i,0}(E^3)^G$, we can compute the irregularities $q_i$ in terms of invariant holomorphic differential forms:  
\[
q_i(X)=q_i(\widehat{X})=\dim_{\mathbb C}\big(H^{i,0}(E^3)^G\big).  
\]
It is common to denote the  top irregularity $q_3$ by $p_g$ and call it the geometric genus  $X$ or $\widehat{X}$, respectively. 
\end{rem}

\begin{rem}\label{InvS}   Let $X:=E^3/G$, where $G$ is a finite group acting diagonally on $E^3$. 

1) According to Proposition  \ref{quadraticdiff}
 rigidity means that none of the quadratic differentials $dz_i \otimes dz_j$ is $G$- invariant. This implies that $q_2(X)=0$, since 
none of the $2$-forms $dz_i \wedge dz_j$ can then be invariant either. By the same reason, we have 
$q_1(X)=0$. 
%In particular the first \emph{Betti number} vanishes $b_1=0$.

2)  If  $dz_1 \wedge dz_2 \wedge dz_3$ is $G$-invariant, then the canonical sheaf of $X$ is trivial, hence $X$ is Gorenstein and $p_g(X)=1$. Otherwise, $p_g(X)=0$, and in this case $\hol(3K_X) \simeq \hol_X$.

3) If $p_g(X)=1$, then  $X$ is a Gorenstein Calabi-Yau threefold and  its  singularities must be of type $\frac{1}{3}(1,1,1)$. 
If the $G$-action is moreover rigid, none of the quadratic differentials $dz_i \otimes dz_j$ is $G$- invariant and  an easy calculation using the invariance of $dz_1 \wedge dz_2 \wedge dz_3$ shows that there are no invariant forms of type $(1,2)$ on $E^3$.
In particular, the topological Euler number is given by 
$e(X)=2  \dim_{\mathbb C}\big( H^{1,1}(E^3)^G\big)$, because $H^i(X,\mathbb C) \simeq H^i(E^3,\mathbb C)^G$ for all $i$. 

4) Similarly,  if $p_g(X)=0$ and $X$ is rigid, then we have 
\[
e(X)=2  \big[1+  \dim_{\mathbb C}\big( H^{1,1}(E^3)^G\big) -  \dim_{\mathbb C}\big( H^{1,2}(E^3)^G\big)\big].
\] 

\end{rem}

\begin{lemma}\label{Euler}
Let $X$ be a quotient of  $E^3$ by a rigid diagonal action of $\mathbb Z_3^2$ or $\He(3)$.
 Let $N_{gor}$ be the number of singularities of type $\frac{1}{3}(1,1,1)$ and $N_{ter}$ be the number of singularities of type $ \frac{1}{3}(1,1,2)$, then 
\[
e(X)=\frac{2}{3}(N_{gor}+N_{ter}). 
\]
\end{lemma}

\begin{proof}
The quotient map  $\pi \colon E^3 \to X$ restricts to an unramified cover 
\[
\pi \colon E^3\setminus \pi^{-1}(S) \to X\setminus S 
\]
of degree $d=9$ or $27$, where $S:=\Sing(X)$. Since the Euler number is additive and $e(E^3)=0$, we obtain: 
\[
-\vert  \pi^{-1}(S) \vert = e\big(E^3\setminus  \pi^{-1}(S)\big)=d e\big(X\setminus S\big) = d\big(e(X) - \vert S \vert \big). 
\]
We conclude the proof, because the fibre of $\pi$ over each singularity  consists of $d/3$ points. 
\end{proof}

\begin{proposition}\label{welcheCalabi}
Let $X$ be a quotient of $E^3$ by a rigid action of $\mathbb Z_3^2$ or $\He(3)$. 
\begin{enumerate}
\item If $p_g=0$ then $X$ has $9$ terminal singularities of type $\frac{1}{3}(1,1,2)$. 
\item
If $p_g=1$ then $\Sing(X)$ consists of  $9$ or $27$  Gorenstein singularites of type $\frac{1}{3}(1,1,1)$. The latter happens if and only if $X$ is isomorphic to Beauville's threefold $X_{3,3}$. 
\end{enumerate}
\end{proposition}

\begin{proof}
1) Thanks to the orbifold Riemann-Roch formula (see e.g.  \cite[ p. 412 Corollary 10.3]{R87}) it is possible to express $ \chi(\mathcal O_X)$ in terms of the intersection of Chern classes of a resolution  $\rho \colon \widehat{X} \ra X$ and local data coming from the singularities. In fact, it reads
\[
24\chi(\mathcal O_X) = - c_1(\rho^*K_X) \cdot c_2(\widehat{X}) + \sum_{\tiny{x ~ ter}} \frac{n_x^2-1}{n_x}, 
\]
where  the sum runs over all terminal singularities $\frac{1}{n_x}(1,a_x,n_x-a_x)$ of $X$. Since $\hol(3K_X) \simeq \hol_X$, the first summand is zero. Moreover,   all terminal singularities of $X$ are  of type $\frac{1}{3}(1,1,2)$  and $\chi(\hol_X) = 1$. Hence the formula implies $N_{ter}=9$.  

2) According to Lemma \ref{Euler} and Remark \ref{InvS} 3), we have 
\[
N_{gor}=3 \dim_{\mathbb C}\big( H^{1,1}(X)^G\big).
\]
As in the proof of Proposition 
\ref{Hodgesmooth} we consider $X$ as a $\mathbb Z_3^2$-quotient of a three dimensional torus und use 
 the 
analytic representation $\rho$ to compute the dimension of the space of invariant $(1,1)$-forms. 
 The 
analytic representation $\rho$ is easy to describe, because it is a  
sum of three non-trivial characters $\chi_i$ of $\mathbb Z_3^2$.  
The invariance of $dz_1 \wedge dz_2 \wedge dz_3$ imposes the condition   $\chi_1 \chi_2 \chi_3 =\chi_{triv}$. Moreover,  we 
 have $\chi_1 \neq \chi_2^2$, because  $dz_1 \otimes dz_2$ is not invariant.  Whence  there are  two cases: 

\noindent 
\underline{Case 1:}  $\chi_1=\chi_2$, then $\chi_3$ is also equal to $\chi_1$ and up to an automorphism of $\mathbb Z_3^2$ it holds 
$\rho(a,b)=\diag\big(\zeta_3^{a},\zeta_3^{a}, \zeta_3^{a} \big)$. In this case all $(1,1)$ forms are invariant i.e., 
$\dim_{\mathbb C}\big( H^{1,1}(X)^G\big)=9$ and we  conclude $N_{gor}=27$. 
The argument we used in  the  proof of Theorem \ref{oneisoclass} i.e. \cite[Corollary 13.3.5]{BL} tells us that $X \simeq X_{3,3}$.

\underline{Case 2:}  $\chi_1 \neq \chi_2$, then since  $\chi_1 \neq \chi_2^2$
there exists an automorphism  of $\mathbb Z_3^2$ such that 
$\chi_1(a,b)=\zeta_3^{a}$ and $\chi_2(a,b)=\zeta_3^{b}$. This implies that  $\chi_3(a,b)=\zeta_3^{2a+2b}$ and shows that  the analytic representation is 
$\rho(a,b)=\diag\big(\zeta_3^{a},\zeta_3^{b}, \zeta_3^{2a+2b} \big)$. 
We compute  $\dim_{\mathbb C}\big( H^{1,1}(X)^G\big)=3$ and obtain $N_{gor}=9$. 
\end{proof}

\begin{proposition}
Let $X$ be a quotient of  $E^3$ by a rigid diagonal action of $G=\mathbb Z_2 \times \mathbb Z_4$ or $\mathbb Z_2^2 \rtimes_{\varphi_4} \mathbb Z_4$, then $X$ has 
non-canonical singularities. 
\end{proposition}
\begin{proof}
We show that $\rho(g)=\zeta_4 \cdot  \Id$, for some $g\in G$, which implies the existence of a singularity of type $\frac{1}{4}(1,1,1)$. 
As above,  we may consider $X$ as a quotient by $G=\mathbb Z_2 \times \mathbb Z_4$. The analytic representation $\rho$ is a sum of three characters 
$\chi_i$ of $G$ of order $4$, such that $\chi_i \neq \overline{\chi_j}$ for $i \neq j$. The order $4$ characters of $G$ are: 
\[
\zeta_4^b, \quad  \zeta_4^{3b}, \quad   (-1)^a  \zeta_4^b \quad \makebox{and}  \quad (-1)^{a} \zeta_4^{3b}.  
\]
Since they  come in pairs of conjugates, we conclude  that 
two of the three characters $\chi_i$ in the analytic representation must be the same. Without loss of generality $\chi_1=\chi_2=\zeta_4^b$ and  $\rho$ is equal to 
\[
\diag(\zeta_4^b,\zeta_4^b,\zeta_4^b), \quad \diag(\zeta_4^b,\zeta_4^b,(-1)^a  \zeta_4^b ) \quad \makebox{or} \quad 
 \diag(\zeta_4^b ,\zeta_4^b,(-1)^{a} \zeta_4^{3b}).
 \]

\end{proof}

Recall that  Beauville's threefold $X_{3,3}$ is simply connected, see Remark \ref{Beau}. This allows us to show: 

\begin{proposition}\label{BeauvilleUnifor}
Let $X$ be a quotient of $E^3$ by a rigid diagonal action of $\mathbb Z_3^2$ or $\He(3)$ with  $p_g=1$ and $9$ singularities. Then 
$X$ is  uniformized by Beauville's threefold $X_{3,3}$ by a degree three map. 
In particular,  $\pi_1(X)$ is isomorphic to $\mathbb Z_3$. 
\end{proposition}

\begin{proof}
We assume that $X$ is a quotient of $E^3$ by $\mathbb Z_3^2$. The case where $X$ is a quotient of $E^3$ by $\He(3)$, or equivalently   $X$ is a quotient of $E^3/C_3$ by $\mathbb Z_3^2$ is handled exactly  in the same way. 

In  the  proof of  Proposition \ref{welcheCalabi} 2), we saw that up to an automorphism of $\mathbb Z_3^2$, the analytic representation is 
$\rho(a,b)=\diag\big(\zeta_3^{a},\zeta_3^{b}, \zeta_3^{2a+2b} \big)$.  The 
restriction  to the subgroup  $H:=\langle (1,1) \rangle \leq \mathbb Z_3^2$ is generated by $\zeta_3 \cdot \Id$, which  implies that $E^3/H$ is isomorphic to $X_{3,3}$, again using \cite[Corollary 13.3.5]{BL}.   
We show that the map $u$ fitting in the diagram 
\[
\begin{xy}
  \xymatrix{
     E^3 \ar[r]^{p_1} \ar[d]_{p_2} &  X \\
  E^3/H \ar[ru]_u & 
  }
\end{xy},
\]
 is a local biholomorphism. Then, since $u$ is proper, it is an unramified  cover. 
Clearly,  $u$ maps $\Sing(E^3/H)=\lbrace 27 \times \frac{1}{3}(1,1,1) \rbrace$ to $\Sing(X)=\lbrace 9 \times \frac{1}{3}(1,1,1) \rbrace$. More precisely,  
since $u$ has degree three,   the fibre of $u$ over each point $p \in X$ consists either of three singular or three smooth points, depending if $p$ is singular or not. 
 Take a point  $q \in E^3/H$ such that $u(q)=p$
and a point $x \in E^3$ such that $p_ 2(x) =q$, 
then  $p_1(x)=p$ and the diagram of local rings  commutes:  
\[
\begin{xy}
  \xymatrix{
     \mathcal O_{E^3,x}  &  \ar[l]_{p_1^{\ast}} \mathcal O_{X,p} \ar[ld]^{u^{\ast}} \\
     \mathcal O_{E^3/H,q} \ar[u]^{p_2^{\ast}} & 
  }
\end{xy},
\]

 By definition of the sheaf of holomorphic functions on a quotient,   the maps $p_1^{\ast}$ resp. $p_2^{\ast}$ are isomorphisms 
onto the subrings $\mathcal O_{E^3,x}^{G_x}$ resp. $\mathcal O_{E^3,x}^{H_x}$ of $\mathcal O_{E^3,x}$.  The inclusion   $\mathcal O_{E^3,x}^{G_x} \subset \mathcal O_{E^3,x}^{H_x}$
is an equality, since   the stabilizers  $H_x \subset G_x$ are equal.  Indeed  both groups  $H_x$ and $G_x$ are either  trivial or of  order $3$, depending if 
 $p$ and $q$ are smooth or singular. 
\end{proof}

\begin{rem}
Let $X$ be a quotient of $E^3$ by a rigid diagonal action of $\mathbb Z_3^2$ or $\He(3)$ uniformized by $X_{3,3}$. 
Then there is a crepant resolution 
$\rho \colon \widehat{X} \to X$ of singularities, such that $\widehat{X}$ is uniformized by the crepant resolution $\widehat{X}_{3,3}$ of Beauville's threefold.
\end{rem}

\begin{theorem}\label{rigsingkod0}
For each exceptional group $\mathbb Z_3^2$ and $\He(3)$ there are exactly 
four isomorphism classes of quotients 
\[
X_i:=E^3/\mathbb Z_3^2 \qquad \makebox{and} \qquad Y_i:=E^3/\He(3)
\]
 obtained by a rigid diagonal $G$-action. Some of the invariants are listed in the following table:
 \begin{center}
{\footnotesize
{
\renewcommand{\arraystretch}{1.5}
\setlength{\tabcolsep}{4pt}
\begin{tabular}{c  c  | c  c  c | c }
 & & $p_g$ & $b_3$ & $b_2$ & $\Sing$  \\
 \hline
\hline
$X_1$ & $Y_1$ & $0$  &  0 & $5$ & $9 \times \frac{1}{3}(1,1,1)$,  $9 \times \frac{1}{3}(1,1,2)$ \\
\hline
$X_2$ & $Y_2$ & $0$ & 2   & $3$& $9 \times \frac{1}{3}(1,1,2)$ \\
\hline
$X_3$ & $Y_3$ & $1$ & 2 & $3$ &  $9 \times \frac{1}{3}(1,1,1)$  \\
\hline
$X_4$ & $Y_4 $ & $1$  & 2 & $9$ &  $27 \times \frac{1}{3}(1,1,1)$ \\
\hline
\end{tabular}
}}
\end{center}
 The threefolds $X_4$ and $Y_4$ are isomorphic to Beauville's threefold. 
\end{theorem}

\begin{proof}
In analogy to the proof of Theorem \ref{Z1Z2Glatt} we consider for each group 
 $G=\mathbb Z_3^2$ resp. $\He(3)$ the set of triples of  generating triples
$[V_1,V_2,V_3]$ 
which correspond to a diagonal rigid group action on $E^3$. 
Then we determine the orbits of the action of  $\mathfrak S_3 \times \mathcal B_3^3 \times \Aut(G)$ on this set. We obtain 
$4$ orbits for each group.
The invariants are then computed as explained in Remark \ref{InvS} and Lemma \ref{Euler}.
\end{proof}
\begin{rem}
1) In the sequel we shall investigate the relations among the threefolds in the table. 
By  looking at the Betti numbers, it is obvious  that $X_i$ is not homeomorphic to   $X_j$ or $Y_j$ for $j\neq i$, except for $i=2$ and $j=3$ or vice versa. 

2)  
Using as in the proof of Theorem  \ref{Z1Z2Glatt} the identification $\He(3)/C_3 \simeq \mathbb Z_3^2$, it can be checked again by a MAGMA routine that the image of the triple $[W_{i,1}, \ldots, W_{i,3}]$ representing $Y_i$ is a triple $[V_{i,1}, \ldots,V_{i,3}]$ which lies in the orbit representing $X_i$. 

Therefore  there are ramified  Galois covers $f_i \colon Y_i \to X_i$ with group $\mathbb Z_3^2$, thanks to Proposition \ref{coverf}.

\end{rem}

\begin{proposition}\label{diffeo}
The threefolds $X_2$ and $X_3$ are diffeomorphic in the orbifold sense. 
Likewise $Y_2$ and $Y_3$ are also diffeomorphic. 
\end{proposition}

\begin{rem}
By Proposition \ref{BeauvilleUnifor} it follows that 
$\pi_1(X_2)$ and $\pi_1(Y_2)$ are isomorphic to $\mathbb Z_3$. 
Therefore we know all fundamental groups of the threefolds $X_i, Y_i$ in the above table, except for $i = 1$. It can be shown using Armstrong's result (cf. \cite{armstrong}) that they are simply connected.
\end{rem}

\begin{proof}[Proof of Proposition \ref{diffeo}]
We may realize $X_3$ as the quotient of $E^3$ by the $\mathbb Z_3^2$-action:
\[
\psi_3(a,b)(z)=\diag(\zeta_3^a,\zeta_3^b,\zeta_3^{2a+2b})z+ \frac{1+2\zeta_3}{3}(b,a,a). 
\]
Indeed the action is faithful on each factor and, according to the argument in the proof of Proposition \ref{welcheCalabi}, Case 2, the quotient has 
$9$ singularities of type $\frac{1}{3}(1,1,1)$. 

Now we modify this action by replacing  the third component of  $\psi_3$ with the complex conjugate:
\[
\psi_2(a,b)(z)=\diag(\zeta_3^a,\zeta_3^b,\zeta_3^{a+b})z+ \bigg( \frac{1+2\zeta_3}{3}b, \frac{1+2\zeta_3}{3}a, \frac{1+2\zeta_3^2}{3}a\bigg). 
\]
The action is  still faithful on each factor, rigid and the invariants of the quotient are $p_g=0$, $b_3=2$ and $b_2=3$. 
Whence the quotient with respect to this action is $X_2$.  By construction, the diffeomorphism 
\[
F\colon E^3 \to E^3, \qquad (z_1,z_2,z_3) \mapsto (z_1,z_2,\overline{z}_3) 
\]
descends to the quotients $\widehat{F} \colon X_3 \to X_2$. An affine diffeomorphism between $Y_3$ and $Y_2$ is established in the same way. 
\end{proof}

\bigskip
The rest of the section is devoted to show the following:

\begin{theorem}\label{nothomeo}
The threefolds in the table consist of five distinct  topological types: 
\[
X_1, \quad  Y_1, \quad  X_2 \simeq_{\makebox{\tiny{diff}}} X_3, \quad  Y_2\simeq_{\makebox{\tiny{diff}}} Y_3 \quad \makebox{and} \quad X_4\simeq_{\makebox{\tiny{bihol}}} Y_4
\]
\end{theorem}

To prove the theorem it remains to show that the threefolds  $X_i$ and $Y_i$ are not homeomorphic for $1 \leq i \leq 3$. 
Albeit we cannot use 
the  fundamental  groups to distinguish them, our  argument  is still 
analogous to the one that we gave in the previous section i.e., based on Bieberbach's theorems. As a substitute for the fundamental group, we use: 

\begin{definition}
Let $T=\mathbb C^n/\Lambda$ be a complex torus and $G$ be a finite group of automorphisms acting on $T$ without 
translations. Let  $\pi \colon \mathbb C^n \to T$ be the universal cover, then we define the \emph{orbifold fundamental group} as
\[
\Gamma:=\lbrace \gamma \colon  \mathbb C^n  \to \mathbb C^n  ~ \big\vert ~ \exists ~ g \in G, ~~ s.t. ~ \pi \circ \gamma = g \circ \pi \rbrace. 
\] 
\end{definition}

\begin{rem}\label{orbipi1} 
 1) By definition $\Gamma$ is  a  cocompact discrete subgroup of the group of affine transformations. The subgroup of translations  of $\Gamma$ is the  lattice $\Lambda$. 
If $G$ acts freely in codimension at least two, then  $\Gamma$ is isomorphic  to the fundamental group of the smooth locus of $T/G$. 

2) We point out that $\Gamma_{X_i}$ and $\Gamma_{Y_i}$ can be described in terms of the triples 
$[V_{i,1}, \ldots,V_{i,3}]$ and  $[W_{i,1}, \ldots, W_{i,3}]$ of generating triples, which correspond to $X_i$ and $Y_i$ (cf. \cite[Section 3]{FourNames}):  let $\mathbb T$ be the triangle group $\mathbb T(3,3,3)$ 
and consider the  homomorphisms
\[
\phi_{V_{i,j}} \colon \mathbb T  \to \mathbb Z_3^2 \qquad \makebox{and} \qquad 
\phi_{W_{i,j}} \colon \mathbb T \to  \He(3). 
\]
Then the groups $\Gamma_{X_i}$ and $\Gamma_{Y_i}$  are isomorphic to the  fibred products:
\begin{align*}
\Gamma_{X_i}& \simeq  \lbrace  t \in  \mathbb T^3 ~ \big\vert ~ \phi_{V_{i,1}}(t_1)= \ldots = \phi_{V_{i,3}}(t_3) \rbrace, \\
\Gamma_{Y_i} & \simeq \lbrace  t \in  \mathbb T^3 ~ \big\vert ~ \phi_{W_{i,1}}(t_1)= \ldots = \phi_{W_{i,3}}(t_3) \rbrace.
\end{align*}

\end{rem}

\begin{lemma}
The groups $\Gamma_{X_i}$ and $\Gamma_{Y_i}$ are not isomorphic for $1\leq i \leq 3$. 
\end{lemma}

\begin{proof}
We only treat the case $i=1$. The strategy in the other cases is the same, even easier. 

Assume that $\Gamma_{X_1}$ and $\Gamma_{Y_1}$ are isomorphic then,   by Bieberbach's second theorem \ref{biberer},   there exists an affine transformation
\[
\alpha \colon \mathbb R^6 \to \mathbb R^6, \quad x \mapsto Ax+b, 
\]
such that $\alpha \circ \Gamma_{X_1} \circ \alpha^{-1}=  \Gamma_{Y_1}$. As explained above, the analytic representations 
of the $\mathbb Z_3^2$ actions giving the quotients $X_1$ and $Y_1$ coincide. We view this representation $\rho=\rho_1$ as a real representation in the orthogonal group of $6\times 6$ matrices:
\[
 \rho_{\mathbb R}(a,b) := 
\begin{pmatrix} B^{a+b}  &  0 & 0   \\ 0 & B^{a+b} & 0  \\ 0 & 0 & B^b 
\end{pmatrix},
\quad B=-\frac{1}{2}
\begin{pmatrix} 1 &  \sqrt{3}  \\ -\sqrt{3} &1 \end{pmatrix}. 
\]
In analogy to Remark \ref{diffeovarphi} (2), there exists 
an  automorphism $\varphi \in \Aut(\mathbb Z_3^2)$, such that
$$
A \rho_{\mathbb R}(a,b) A^{-1} =\rho_{\mathbb R}\big(\varphi(a,b)\big), \quad  \ \forall \  (a,b) \in \mathbb Z_3^2. 
$$
The representation $\rho_{\mathbb R}$ consists of  two copies of  $B^{a+b} $ and one copy of $B^{b}$. Thus $\rho_{\mathbb R} \circ \varphi$ is also the sum 
of two  copies of an irreducible two-dimensional real representation and another (distinct)  irreducible two dimensional real representation.  
Schur's Lemma implies that $A$ is a block matrix: 
\[
A= 
\begin{pmatrix} A_1 &  0   \\ 0 & A_2  
\end{pmatrix}, \quad \makebox{where} \quad A_1 \in \GL(4,\mathbb R) \quad \makebox{and} \quad A_2 \in \GL(2,\mathbb R).
\]
As in the proof of Theorem \ref{Z1notZ2} we conclude that  $A_2$ has to be $\CC$-linear or $\CC$-antilinear and we obtain a contradiction for the same reasoning. 
\end{proof}

\begin{rem}
We can use the description of the groups $\Gamma_{X_i}$ and $\Gamma_{Y_i}$ from Remark \ref{orbipi1} (2) to compute,  with the help of MAGMA,  
 the number of their index three  normal subgroups: 
 \begin{center}
{\footnotesize
\begin{tabular}{c | c | c | c | c| c| c }
  & $\Gamma_{X_1}$  & $\Gamma_{Y_1}$ & $\Gamma_{X_2}$ & $\Gamma_{Y_2}$ & $\Gamma_{X_3}$ & $\Gamma_{Y_3}$ \\
\hline
\# subgrps  & $41$  & $14$ & $41$ &  $5$ & $41$ &  $5$ \\
\end{tabular}
}
\end{center}
This is provides another argument  that the groups  $\Gamma_{X_i}$ and $\Gamma_{Y_i}$ cannot be isomorphic for $1 \leq i \leq 3$. 
Note that Proposition \ref{diffeo} implies $\Gamma_{X_2} \simeq \Gamma_{X_3}$ and  $\Gamma_{Y_2}\simeq \Gamma_{Y_3}$, a fact that can be also verified 
using the MAGMA command {\tt SearchForIsomorphism}. 
\end{rem}

\begin{proof}[Proof of Theorem \ref{nothomeo}]
It remains to show that the threefolds  $X_i$ and $Y_i$ are not homeomorphic for $1 \leq i \leq 3$. 
Assume the converse. We claim that a 
homeomorphism $f_i \colon X_i \to Y_i$ maps smooth to smooth and 
 singular to  singular points. In particular, it restricts to a homeomorphism between the regular  loci $f_i \colon  X_i^{\circ} \to Y_i^{\circ}$ and therefore  induces  an isomorphism between 
 $\Gamma_{X_i}$ and $\Gamma_{Y_i}$, see Remark \ref{orbipi1} (1).  A contradiction.  
 To verify  the claim,  we point out that the local fundamental group of a singularity of type $\frac{1}{3}(1,1,1)$ or $\frac{1}{3}(1,1,2)$ is isomorphic to 
 $\mathbb Z_3$, while the local fundamental group of a smooth point is  trivial. 
We conclude the proof,  because $f$ induces an isomorphism between 
$\pi_1^{loc}(X_i,p)$ and  $\pi_1^{loc}\big(Y_i,f_i(p)\big)$,  
for all $p \in X_i$. 
\end{proof}

{\tiny MATHEMATISCHES INSTITUT, UNIVERSIT\"AT BAYREUTH, 95440 BAYREUTH, GERMANY}

{\scriptsize\emph{E-mail address: Ingrid.Bauer@uni-bayreuth.de}} \quad {\scriptsize\emph{Christian.Gleissner@uni-bayreuth.de}}


\bigskip
\bigskip

\begin{biblist}

\bib{armstrong}{article}{
    AUTHOR = {Armstrong, M. A.},
     TITLE = {The fundamental group of the orbit space of a discontinuous
              group},
   JOURNAL = {Proc. Cambridge Philos. Soc.},
  FJOURNAL = {Proceedings of the Cambridge Philosophical Society},
    VOLUME = {64},
      YEAR = {1968},
     PAGES = {299--301},
      ISSN = {0008-1981},
   MRCLASS = {54.80},
  MRNUMBER = {221488},
MRREVIEWER = {R. W. Bagley},
       DOI = {10.1017/s0305004100042845},
       URL = {https://doi.org/10.1017/s0305004100042845},
}
%
\bib{rigidity}{article}{
   author={Bauer, Ingrid},
   author={Catanese, Fabrizio},
   title={On rigid compact complex surfaces and manifolds},
   journal={Adv. Math.},
   volume={333},
   date={2018},
   pages={620--669},
   issn={0001-8708},
   review={\MR{3818088}},
   doi={10.1016/j.aim.2018.05.041},
}
%
\bib{BeauReal}{article}{
     author={Bauer, Ingrid},
     author={Catanese, Fabrizio},
     author={Grunewald, Fritz},
       title={Beauville surfaces without real structures},
 booktitle= {Geometric methods in algebra and number theory},
    series = {Progr. Math.},
    volume = {235},
    pages = {1--42},
      date = {2005},
  review={\MR{2159375}},
}
%
\bib{FourNames}{article}{
     author={Bauer, Ingrid},
     author={Catanese, Fabrizio},
     author={Grunewald, Fritz},
      author={Pignatelli, Roberto},    
     TITLE = {Quotients of products of curves, new surfaces with {$p_g=0$}
              and their fundamental groups},
   JOURNAL = {Amer. J. Math.},
  FJOURNAL = {American Journal of Mathematics},
    VOLUME = {134},
      YEAR = {2012},
    NUMBER = {4},
     PAGES = {993--1049},
      ISSN = {0002-9327},
   MRCLASS = {14J29 (14J10)},
  MRNUMBER = {2956256},
MRREVIEWER = {Christian Liedtke},
       DOI = {10.1353/ajm.2012.0029},
       URL = {https://doi.org/10.1353/ajm.2012.0029},
}
%	
\bib{BG}{article}{
    AUTHOR = {Bauer, Ingrid},
    author={Gleissner, Christian},
     TITLE = {Fermat's cubic, {K}lein's quartic and rigid complex manifolds
              of {K}odaira dimension one},
   JOURNAL = {Doc. Math.},
  FJOURNAL = {Documenta Mathematica},
    VOLUME = {25},
      YEAR = {2020},
     PAGES = {1241--1262},
      ISSN = {1431-0635},
   MRCLASS = {14B12 (14D06 14M25 32G05 32G07)},
  MRNUMBER = {4164723},
       DOI = {10.3934/dcdsb.2019218},
       URL = {https://doi.org/10.3934/dcdsb.2019218},
}
\bib{notinfinitesimally}{article}{
   author={Bauer, Ingrid},
   author={Pignatelli, Roberto},
    title={Rigid but not infinitesimally rigid compact complex manifolds},
   eprint={arXiv:1805.02559 [math.AG]},
   date={2018},
    pages={18}
}
%
%
\bib{beauville}{article}{
   author={Beauville, Arnaud},
   title={Some remarks on K\"ahler manifolds with $c_{1}=0$},
   conference={
      title={Classification of algebraic and analytic manifolds},
      address={Katata},
      date={1982},
   },
   book={
      series={Progr. Math.},
      volume={39},
      publisher={Birkh\"auser Boston, Boston, MA},
   },
   date={1983},
   pages={1--26},
   review={\MR{728605}},
   doi={10.1007/BF02592068},
}
%
\bib{bib1}{article}{
    AUTHOR = {Bieberbach, Ludwig},
     TITLE = {\"{U}ber die {B}ewegungsgruppen der {E}uklidischen {R}\"{a}ume},
   JOURNAL = {Math. Ann.},
  FJOURNAL = {Mathematische Annalen},
    VOLUME = {70},
      YEAR = {1911},
    NUMBER = {3},
     PAGES = {297--336},
      ISSN = {0025-5831},
   MRCLASS = {DML},
  MRNUMBER = {1511623},
       DOI = {10.1007/BF01564500},
       URL = {https://doi.org/10.1007/BF01564500},
}
%		
\bib{bib2}{article}{
    AUTHOR = {Bieberbach, Ludwig},
     TITLE = {\"{U}ber die {B}ewegungsgruppen der {E}uklidischen {R}\"{a}ume
              ({Z}weite {A}bhandlung.) {D}ie {G}ruppen mit einem endlichen
              {F}undamentalbereich},
   JOURNAL = {Math. Ann.},
  FJOURNAL = {Mathematische Annalen},
    VOLUME = {72},
      YEAR = {1912},
    NUMBER = {3},
     PAGES = {400--412},
      ISSN = {0025-5831},
   MRCLASS = {DML},
  MRNUMBER = {1511704},
       DOI = {10.1007/BF01456724},
       URL = {https://doi.org/10.1007/BF01456724},
}
%
\bib{BL}{book}{
    AUTHOR = {Birkenhake, Christina},
    AUTHOR= {Lange, Herbert},
     TITLE = {Complex abelian varieties},
    SERIES = {Grundlehren der Mathematischen Wissenschaften [Fundamental
              Principles of Mathematical Sciences]},
    VOLUME = {302},
   EDITION = {Second},
 PUBLISHER = {Springer-Verlag, Berlin},
      YEAR = {2004},
     PAGES = {xii+635},
      ISBN = {3-540-20488-1},
   MRCLASS = {14-02 (14H37 14Kxx 32G20)},
  MRNUMBER = {2062673},
MRREVIEWER = {Fumio Hazama},
       DOI = {10.1007/978-3-662-06307-1},
       URL = {https://doi.org/10.1007/978-3-662-06307-1},
}
%
\bib{MAGMA}{article}{
    AUTHOR = {Bosma, Wieb},
    AUTHOR = {Cannon, John},
    AUTHOR = { Playoust, Catherine},        
     TITLE = {The {M}agma algebra system. {I}. {T}he user language},
      NOTE = {Computational algebra and number theory (London, 1993)},
   JOURNAL = {J. Symbolic Comput.},
  FJOURNAL = {Journal of Symbolic Computation},
    VOLUME = {24},
      YEAR = {1997},
    NUMBER = {3-4},
     PAGES = {235--265},
      ISSN = {0747-7171},
   MRCLASS = {68Q40},
  MRNUMBER = {MR1484478},
       DOI = {10.1006/jsco.1996.0125},
       URL = {http://dx.doi.org/10.1006/jsco.1996.0125},
}
%
\bib{AndiFab}{article}
{
    AUTHOR = {Catanese, Fabrizio},
    AUTHOR = {Demleitner, Andreas},
     TITLE = {The classification of hyperelliptic threefolds},
   JOURNAL = {Groups Geom. Dyn.},
  FJOURNAL = {Groups, Geometry, and Dynamics},
    VOLUME = {14},
      YEAR = {2020},
    NUMBER = {4},
     PAGES = {1447--1454},
      ISSN = {1661-7207},
   MRCLASS = {Prelim},
  MRNUMBER = {4186481},
       DOI = {10.4171/ggd/587},
       URL = {https://doi.org/10.4171/ggd/587},
}
%
\bib{LCh}{book}{
    AUTHOR = {Charlap, Leonard S.},
     TITLE = {Bieberbach groups and flat manifolds},
    SERIES = {Universitext},
 PUBLISHER = {Springer-Verlag, New York},
      YEAR = {1986},
     PAGES = {xiv+242},
      ISBN = {0-387-96395-2},
   MRCLASS = {57S30 (22E40 53C30)},
  MRNUMBER = {862114},
MRREVIEWER = {Kyung Bai Lee},
       DOI = {10.1007/978-1-4613-8687-2},
       URL = {https://doi.org/10.1007/978-1-4613-8687-2},
       }
%  
\bib{Demleitner}{thesis}{
    AUTHOR = {Demleitner, Andreas},
     TITLE = {On Hyperelliptic Manifolds, PhD thesis University of Bayreuth},
   FJOURNAL = {EPub Bayreuth, PhD thesis, University of Bayreuth},
      YEAR = {2020},
     PAGES = {0--201},
}
%
\bib{frob}{article}{
    AUTHOR = {Frobenius, Ferdinand Georg},
     TITLE = {\"Uber lineare Substitutionen und bilineare Formen},
   JOURNAL = {J. Reine Angew. Math.},
  FJOURNAL = {Journal f\"ur die reine und angewandte Mathematik},
    VOLUME = {84},
      Date = {1877},
       pages={1--63},}
%

%
\bib{Lange}{article}{
    AUTHOR = {Lange, Herbert},
     TITLE = {Hyperelliptic varieties},
   JOURNAL = {Tohoku Math. J. (2)},
  FJOURNAL = {The Tohoku Mathematical Journal. Second Series},
    VOLUME = {53},
      YEAR = {2001},
    NUMBER = {4},
     PAGES = {491--510},
      ISSN = {0040-8735},
   MRCLASS = {14J30 (14J50)},
  MRNUMBER = {1862215},
MRREVIEWER = {Miguel A. Barja},
       DOI = {10.2748/tmj/1113247797},
       URL = {https://doi.org/10.2748/tmj/1113247797},
}
%
\bib{kodairamorrow}{book}{
   author={Morrow, James},
   author={Kodaira, Kunihiko},
   title={Complex manifolds},
   publisher={Holt, Rinehart and Winston, Inc., New York-Montreal,
   Que.-London},
   date={1971},
   pages={vii+192},
   review={\MR{0302937}},
}
%
\bib{miranda}{book}{
    AUTHOR = {Miranda, Rick},
     TITLE = {Algebraic curves and {R}iemann surfaces},
    SERIES = {Graduate Studies in Mathematics},
    VOLUME = {5},
 PUBLISHER = {American Mathematical Society, Providence, RI},
      YEAR = {1995},
     PAGES = {xxii+390},
      ISBN = {0-8218-0268-2},
   MRCLASS = {14Hxx (14-01 30F99)},
  MRNUMBER = {1326604},
MRREVIEWER = {R. F. Lax},
       DOI = {10.1090/gsm/005},
       URL = {https://doi.org/10.1090/gsm/005},
}
%
\bib{R87}{article}{
   author={Reid, Miles},
   title={Young person's guide to canonical singularities},
   conference={
      title={Algebraic geometry, Bowdoin, 1985},
      address={Brunswick, Maine},
      date={1985},
   },
   book={
      series={Proc. Sympos. Pure Math.},
      volume={46},
      publisher={Amer. Math. Soc., Providence, RI},
   },
   date={1987},
   pages={345--414},
   review={\MR{927963}},
}
%	
\bib{schlessinger}{article}{
    AUTHOR = {Schlessinger, Michael},
     TITLE = {Rigidity of quotient singularities},
   JOURNAL = {Invent. Math.},
  FJOURNAL = {Inventiones Mathematicae},
    VOLUME = {14},
      YEAR = {1971},
     PAGES = {17--26},
      ISSN = {0020-9910},
   MRCLASS = {32G05},
  MRNUMBER = {292830},
MRREVIEWER = {F. Oort},
       DOI = {10.1007/BF01418741},
       URL = {https://doi.org/10.1007/BF01418741},
}
%
\bib{UchidaYoshi}{article}{
    AUTHOR = {Uchida, K\^{o}ji},
    AUTHOR = {Yoshihara, Hisao},
     TITLE = {Discontinuous groups of affine transformations of {$C^{3}$}},
   JOURNAL = {Tohoku Math. J. (2)},
  FJOURNAL = {The Tohoku Mathematical Journal. Second Series},
    VOLUME = {28},
      YEAR = {1976},
    NUMBER = {1},
     PAGES = {89--94},
      ISSN = {0040-8735},
   MRCLASS = {57E30},
  MRNUMBER = {400271},
MRREVIEWER = {F. A. Sherk},
       DOI = {10.2748/tmj/1178240881},
       URL = {https://doi.org/10.2748/tmj/1178240881},
}

\end{biblist}
\end{document}